\documentclass[12pt]{amsart}

\usepackage{amsmath,amsthm,amssymb}
\usepackage{tikz}
\usetikzlibrary{arrows.meta}

\usepackage{a4wide}
\usepackage{enumerate}
\usepackage[all]{xy}
\usepackage{multicol}

\usepackage{hyperref}
\newtheorem{theorem}{Theorem}[section]
\newtheorem{lemma}[theorem]{Lemma}
\newtheorem{proposition}[theorem]{Proposition}
\newtheorem{corollary}[theorem]{Corollary}
\newtheorem*{MainTheorem}{Main Theorem}
\newtheorem*{claim*}{Claim}
\newtheorem*{fact*}{Fact}

\theoremstyle{remark}
\newtheorem{remark}[theorem]{Remark}

\newcommand{\C}{\ensuremath{\mathbb{C}}}
\newcommand{\R}{\ensuremath{\mathbb{R}}}

\newcommand{\g}[1]{\ensuremath{\mathfrak{#1}}}
\newcommand{\II}{\ensuremath{I\!I}}
\DeclareMathOperator{\tr}{tr}

\DeclareMathOperator{\id}{id}

\DeclareMathOperator{\ad}{ad}
\DeclareMathOperator{\Exp}{Exp}

\DeclareMathOperator{\spann}{span}

\DeclareMathOperator{\Ric}{Ric}
\DeclareMathOperator{\rank}{rank}

\newcommand{\Ss}{\ensuremath{\mathcal{S}}}

\begin{document}
\title[]{Codimension one Ricci soliton subgroups of nilpotent Iwasawa groups}

\author[V.~Sanmart\'in-L\'opez]{V\'ictor Sanmart\'in-L\'opez}
\address{Department of Mathematics, Universidade de Santiago de Compostela}
\email{victor.sanmartin@usc.es}

\begin{abstract}
Any expanding homogeneous Ricci soliton (in particular any homogeneous Einstein manifold of negative scalar curvature) can be obtained, up to isometry, from a Lie subgroup of a nilpotent Iwasawa group $N$ whose induced metric is a Ricci soliton.	By nilpotent Iwasawa group we mean the nilpotent Lie group $N$ of the Iwasawa decomposition associated with a symmetric space of non-compact type. Motivated by this fact, in this paper we classify codimension one Lie subgroups of any nilpotent Iwasawa group $N$ whose induced metric is a Ricci soliton.

\end{abstract}

\thanks{The author has been supported by projects PID2019-105138GB-C21/AEI/10.13039/501100011033 (Spain) and ED431C 2019/10, ED431F 2020/04 (Xunta de Galicia, Spain).} 

\subjclass[2020]{53C40, 53C35, 53C42}
\keywords{Nilsoliton, Einstein solvmanifold, Iwasawa group, symmetric space}
\maketitle
\section{Introduction}\label{section:introduction}
The investigation of Einstein metrics constitutes a classical area of research in differential geometry and general relativity. Despite this fact, a satisfactory general understanding of such metrics has not been achieved yet. This is probably due to the subtlety of the Einstein condition, which is encoded in a non-linear second order PDE. Thus, the nature of the approach, techniques and results in the general context (see survey~\cite{An}) differs significantly from the homogeneous setting, where the aforementioned PDE can be translated into algebraic equations (see surveys~\cite{Wa12},~\cite{La09} or~\cite{Jab:survey}).

One of the driving questions in the area over the last four decades, concerning the homogeneous setting, seems to be the Alekseevskii conjecture~\cite[Conjecture~7.57]{Besse}, recently proved to be true by Böhm and Lafuente~\cite[Theorem~D]{BoLa21}. Combining the original statement with~\cite[Proposition~3.1]{LaLa14} and~\cite[Theorem A]{BoLa18}, one gets: any connected homogeneous Einstein manifold of negative scalar curvature is isometric to a simply connected \emph{Einstein solvmanifold}, i.e. a solvable Lie group endowed with an Einstein left invariant metric.

The most important generalization of Einstein metric is that of Ricci soliton. In this paper, we focus on a remarkable subclass of the latter: algebraic Ricci solitons. A Lie group $S$ endowed with a left invariant metric is said to be an \emph{algebraic Ricci soliton} if there exists a derivation $D$ of the Lie algebra of $S$ and a real number $c$ such that the $(1, 1)$-Ricci tensor of $S$, $\Ric$, reads as
\begin{equation}\label{equation:nilsoliton}
\Ric = D + c \id.
\end{equation}
If $S$ is a solvable (respectively a nilpotent) algebraic Ricci soliton, then it is called a \emph{solvsoliton} (respectively a \emph{nilsoliton}). Any algebraic Ricci soliton is always an example of a homogeneous Ricci soliton~\cite{La11C}. 

\emph{In this article, we aim at investigating the intriguing interplay between submanifold geometry of symmetric spaces and homogeneous Einstein or Ricci soliton metrics.} There are strong reasons to address this line of research. On the one hand, our main source of motivation comes from the submanifold theory viewpoint. We are interested in constructing, describing and classifying submanifolds of symmetric spaces with a high degree of symmetry, as it is the case of homogeneous submanifolds. They have been deeply studied in the literature under certain extra assumptions: having codimension one~\cite{BT03, BT07, BT13, DDR21}, producing a hyperpolar foliation~\cite{BDT10}, or having constant principal curvatures~\cite{BS18}, among others. There also exist several works combining submanifold theory with the Einstein or Ricci soliton condition. For instance, in~\cite{Ta11}, Tamaru provided a large family of Einstein solvmanifolds arising as minimal homogeneous submanifolds of symmetric spaces of non-compact type. In~\cite{NiPa21}, Nikolayevsky and Park classified Einstein hypersurfaces of irreducible symmetric spaces with $\rank \geq 2$, while the classification for rank one was already known (see~\cite{Fi38} for spaces of constant curvature, and~\cite[Theorem~1]{NiPa21} for a summary of the results in the rest of the rank one symmetric spaces). Moreover, in a recent paper by Domínguez-Vázquez, Tamaru and the author~\cite{DST20}, we classified codimension one Lie subgroups of symmetric spaces of non-compact type that are Ricci solitons with the induced metric. 

On the other hand, there is a strong incentive for investigating a particular class of Lie subgroups associated with symmetric spaces of non-compact type that are nilsolitons with the induced metric. Let us clarify this claim. Let $M$ be a symmetric space of non-compact type and $G = KAN$ the Iwasawa decomposition of the identity component of its isometry group $G$. Note that $AN$ is a solvable Lie group constructed as the semidirect product of an abelian Lie group $A$ and a nilpotent group $N$ (see Section~\ref{section:preliminaries} for further details). When $AN$ is endowed with the suitable metric, which turns out to be left invariant, it is isometric to $M$. From now on, nilpotent Lie groups $N$ obtained from the Iwasawa decomposition of the identity component of isometry groups of irreducible symmetric spaces of non-compact type will be called \emph{nilpotent Iwasawa groups}. Surprisingly, we have the following 
\begin{fact*}
Any connected expanding homogeneous Ricci soliton can be obtained, up to isometry, as an extension of a Ricci soliton Lie subgroup of a nilpotent Iwasawa group.
\end{fact*}
This fact is probably known, but since we did not find it explicitly stated in the literature, we include a more precise statement and a detailed proof in Section~\ref{section:nilsolitons}, essentially based on remarkable results by Lauret~\cite{La11C} and Jablonski~\cite{Jab:arxiv}. Note that the expanding case is the only one allowing non-trivial examples in the homogeneous setting, as homogeneous steady Ricci solitons are Ricci flat and then flat~\cite{AlKi}, and the study of homogeneous shrinking Ricci solitons is subsumed into the investigation of compact Einstein manifolds~\cite{Iv93, Na10, PeWy09}. 

Therefore, and according to the above motivations, we have started a long term project aiming at investigating the interplay between submanifold geometry of symmetric spaces and homogeneous Einstein or Ricci solitons metrics, based on the following novel approach: exploring Lie subgroups of nilpotent Iwasawa groups that are nilsolitons when considered with the induced metric. In this article, we focus on \emph{the classification of codimension one Lie subgroups of nilpotent Iwasawa groups that are Ricci solitons with the induced metric} (in a nilpotent Lie group endowed with a left invariant metric, being a Ricci soliton is equivalent to being a nilsoliton~\cite[Proposition~1.1]{La01}). This paper can be also regarded as one of the first steps of this long term project, as we have already produced an intriguing wide variety of Ricci soliton Lie subgroups (of different codimensions) of nilpotent Iwasawa groups, which we will present in a forthcoming paper. 

Moreover, a natural strategy to obtain classifications of expanding homogeneous Ricci solitons is to start by inspecting low dimensional cases, as in~\cite{ArLa15, ArLa17} (Einstein case) or in~\cite[Section~6]{La11C} (solvsoliton case). However, a classification of expanding homogeneous Ricci solitons relies on a classification of nilsolitons (see Fact), and as mentioned in~\cite{La09}, this seems hopeless nowadays. Hence, and according to the submanifold viewpoint that we propose in this article, it is reasonable to investigate solvsolitons starting by these whose nilradical appears with low codimension in a nilpotent Iwasawa group. In this line, the main result of this article can be regarded as \emph{a description of the solvsolitons (expanding homogeneous Ricci solitons, up to isometry) whose nilradical can be isometrically and isomorphically embedded as a codimension one Lie subgroup of a nilpotent Iwasawa group}. 

In order to state the main result of this paper, we introduce very briefly some notations concerning symmetric spaces of non-compact type. Let $M = G/K$ be an irreducible symmetric space of non-compact type, where $G$ stands for the identity component of the isometry group of $M$ and $K$ is the isotropy subgroup of $G$ at an arbitrary but fixed point of $M$. Let $\g{k}$ be the Lie algebra of $K$ and put $\g{g} = \g{k} \oplus \g{p}$ for the corresponding Cartan decomposition of the Lie algebra $\g{g}$ of $G$. Take a maximal abelian subspace $\g{a}$ of $\g{p}$ and let $\g{g} = \g{g}_{0} \oplus \left( \bigoplus_{\alpha \in \Sigma} \g{g}_{\alpha} \right)$ be the induced restricted root space decomposition of $\g{g}$, where $\Sigma$ denotes the set of restricted roots. Choose a notion of positivity in $\Sigma$ and denote by $\Pi$ a set of simple roots according to this choice. Let $\g{g} = \g{k} \oplus \g{a} \oplus \g{n}$ be the corresponding Iwasawa decomposition of $\g{g}$. Then, take $AN$ as the connected Lie subgroup of $G$ with Lie algebra $\g{a} \oplus \g{n}$. Now, $M$ turns out to be isometric to the solvable Lie group $AN$, when endowed with a suitable left invariant metric. The main aim of this work is to prove the following
\begin{MainTheorem}\label{main:theorem}
Let $N$ be the nilpotent group of the Iwasawa decomposition of the identity component of the isometry group of an irreducible symmetric space of non-compact type $M \cong AN$. Let $S$ be a codimension one Lie subgroup of $N$, and let $\xi$ be a left invariant unit vector field on $N$ normal to $S$. Then, $S$ is a Ricci soliton when endowed with the induced metric if and only if one of the following conditions holds:
\begin{enumerate}[{\rm (i)}]
\item $M$ is a hyperbolic space $\R H^{n}$, $\C H^n$, $n \geq 2$, or a hyperbolic plane $\mathbb{H} H^2$ or $\mathbb{O} H^2$. \label{main:theorem:1}
\item $\rank M \geq 2$ and $\g{g}_\alpha = \R \xi$ for some $\alpha \in \Pi$. \label{main:theorem:2} 
\item $M$ is $SL_3(\mathbb{H})/Sp_3$ or $E^{-26}_6/F_4$, and $\xi \in \g{g}_\alpha$ for some $\alpha \in \Pi$. \label{main:theorem:3}
\item $M$ is $SO_5 (\C)/SO_5$ or $SO^0_{2,2+n}/SO_2 SO_{2+n}$, $n \geq 2$, and $\xi \in \g{g}_\alpha$, where $\alpha$ is the shortest root in $\Pi$.\label{main:theorem:4}
\item $M$ is $SL_3 (\R)/SO_3$,  $SL_3 (\C)/SU_3$ or $SO^{0}_{2,3}/SO_2 SO_3$. \label{main:theorem:5}
\item $M$ is $SL_4 (\R)/SO_4$, $SL_4 (\C)/SU_4$ or $SL_4 (\mathbb{H})/Sp_4$, and $\xi = 2^{-1/2} (\xi_\alpha + \xi_\beta)$, for some unit vectors $\xi_\alpha\in\g{g}_\alpha$ and $\xi_\beta\in\g{g}_\beta$, and being $\alpha$, $\beta$ the orthogonal simple roots of~$\Pi$. \label{main:theorem:6}
\end{enumerate}
\end{MainTheorem}

\begin{remark}
It is well-known that non-abelian nilpotent Lie algebras do not admit Einstein metrics~\cite{Mi76}. Hence, the only Einstein examples in the Main Theorem correspond to codimension one Lie subgroups of nilpotent Iwasawa groups of $\R H^n$, $\C H^2$ and $SL(3,\R)/SO(3)$.
\end{remark}

In the following lines, we include some explanations about the examples of the Main Theorem. Concerning item~(\ref{main:theorem:1}), $S$ is a codimension one Lie subgroup of a generalized Heisenberg group. The classification of such $S$ whose induced metric is a Ricci soliton was achieved in~\cite[Theorem B]{DST20}. In item~(\ref{main:theorem:2}), we recover the nilradical of a very particular subclass of the remarkable examples of Einstein minimal solvmanifolds introduced by Tamaru~\cite{Ta11}. Let $\Phi$ be an arbitrary subset of the set $\Pi$ of simple roots. Tamaru proved that the solvable part $A_{\Phi} N_{\Phi}$ of a parabolic subgroup of $G$ is an Einstein solvmanifold. Consequently, one gets that its nilradical, $N_{\Phi}$, is a Ricci soliton Lie subgroup of $N$~\cite[Theorem~4.8]{La11C}. Since in this paper we focus on codimension one Ricci soliton subgroups of $N$, we only recover the nilradicals of the examples in~\cite{Ta11} with $\Phi = \{ \alpha \}$, for some $\alpha \in \Pi$ satisfying $\dim \g{g}_\alpha =1$ (see Proposition~\ref{proposition:1:root}). Note that there is an intersection between item~(\ref{main:theorem:2}) and item~(\ref{main:theorem:5}) that happens when $\xi$ spans one simple root space in~$SL_3 (\R)/SO_3$ or in~$SO^{0}_{2,3}/SO_2 SO_3$.

From item~(\ref{main:theorem:3}) to~(\ref{main:theorem:6}) in the Main Theorem, we get new examples of Ricci solitons arising as Lie subgroups of symmetric spaces of non-compact type. Note that we just have examples in: rank two symmetric spaces, from item~(\ref{main:theorem:3}) to item~(\ref{main:theorem:5}); and in rank three symmetric spaces in item~(\ref{main:theorem:6}). Indeed, a direct consequence of the Main Theorem is that, apart from the nilradical of a very particular class of the examples due to Tamaru~\cite{Ta11}, there is a strong constraint in the rank of $M$ for Ricci solitons to arise as codimension one Lie subgroups of nilpotent Iwasawa groups. More precisely, let us say that a symmetric space of non-compact type $M$ is \emph{split} if $\dim \g{g}_\alpha = 1$ for all $\alpha \in \Sigma^{+}$, and then we have:

\begin{corollary}
If the nilpotent Iwasawa group $N$ of a non-split irreducible symmetric space of non-compact type $M$ admits a codimension one Lie subgroup that is a Ricci soliton when considered with the induced metric, then $\rank M \leq 3$.
\end{corollary}
It is also interesting to remark the lack of new examples in exceptional symmetric spaces of non-compact type, as we make precise in the following

\begin{corollary}
Let $S$ be a codimension one Ricci soliton Lie subgroup of the nilpotent Iwasawa group of an irreducible exceptional symmetric space of non-compact type $M$. Then, $M$ is split and $S$ corresponds to the nilradical of Tamaru's examples in Main Theorem~(\ref{main:theorem:2}). 
\end{corollary}
In contrast to this constraint on the rank for the new examples, we get non-abelian examples of any prescribed odd dimension. More precisely, from Main Theorem~(\ref{main:theorem:4})-(\ref{main:theorem:5}) we deduce the following 

\begin{corollary}
For any prescribed odd dimension $2n+1$, $n \geq 1$, there exists a codimension one Lie subgroup $S$ of the nilpotent Iwasawa group of the complex hyperbolic quadric $SO^0_{2,2+n}/SO_2 SO_{2+n}$ that is a non-abelian Ricci soliton (of dimension $2n+1$) with the induced metric.
\end{corollary}
The degree of nilpotency that we get in the new examples, those from item~(\ref{main:theorem:3}) to item~(\ref{main:theorem:6}), is at most three. Indeed, we have:

\begin{corollary}
Let $S$ be a nilsoliton realized as a Lie subgroup of the nilpotent Iwasawa group $N$ of an irreducible symmetric space of non-compact type. If the degree of nilpotency of $S$ is $\geq 4$, then it corresponds to the nilradical of Tamaru's examples in Main Theorem~(\ref{main:theorem:2}) or its codimension in $N$ is at least $2$.
\end{corollary}

Any codimension one Lie subgroup of the nilpotent Iwasawa group of the symmetric spaces in Main Theorem~(\ref{main:theorem:5}) is a Ricci soliton with the induced metric. Actually, we can guarantee the existence of a continuous family of mutually non-congruent Ricci soliton submanifolds for each one of these symmetric spaces (see Remark~\ref{remark:congruency}).

Along the proof of the Main Theorem, we see that any codimension one Lie subgroup $S$ of a nilpotent Iwasawa group $N$ is minimal (Lemma~\ref{lemma:simplify:ric:S}~(\ref{lemma:simplify:ric:S:3})). Moreover, $S$ is a normal subgroup of $N$ (see Remark~\ref{remark:normal:subgroup}), and hence we have the following

\begin{corollary}
Let $S$ be a codimension one Lie subgroup of the nilpotent Iwasawa group $N$ of an irreducible symmetric space of non-compact type. Then, the foliation of $N$ by left cosets of $S$ is a homogeneous foliation by mutually congruent minimal hypersurfaces. In particular, each of the examples in the Main Theorem induces a foliation of $N$ by mutually congruent Ricci soliton minimal hypersurfaces of $N$.
\end{corollary}

This paper is organized as follows. Section~\ref{section:preliminaries} is devoted to introducing the main tools used for our investigations, concerning symmetric spaces of non-compact type. In Section~\ref{section:nilsolitons}, we include a detailed argument of the fact that any connected expanding homogeneous Ricci soliton can be obtained, up to isometry, as an extension of a Lie subgroup of a nilpotent Iwasawa group that is a Ricci soliton (equivalently, a nilsoliton) with the induced metric. Although we did not find this claim in the literature, our proof is based on a combination of well-known results in the area. In Section~\ref{section:general:setting} we establish the general setting of this work. More precisely, we start by giving a general and suitable description of the Lie algebra $\g{s}$ of a codimension one Lie subgroup $S$ of $N$ and we calculate the shape operator and the mean curvature vector of $S$ as a submanifold of $N$. After that, Section~\ref{section:ricci:operator} is completely devoted to the explicit calculation of the Ricci tensor of $S$. In order to do so, we examine systematically the information that is codified in the restricted root systems of symmetric spaces of non-compact type. By making extensive use of these calculations, in Section~\ref{section:classification} we see that if $S$ is a Ricci soliton, then: the left invariant unit vector field normal to $S$ in $N$ belongs to one simple root space; or $S$ is a Lie subgroup of the nilpotent Iwasawa group $N$ associated with one of the irreducible symmetric spaces of non-compact type from item~(\ref{main:theorem:3}) to item~(\ref{main:theorem:6}) of the Main Theorem. In other words, in Section~\ref{section:classification}, the classification problem is reduced to certain symmetric spaces of non-compact type of rank less or equal than three, or to codimension one Lie subgroups of $N$ with unit normal vector in a simple root space. Finally, in Section~\ref{section:main:theorem} we check that all the examples of the Main Theorem are indeed Ricci solitons and we finish the proof of the classification result.

I would like to thank J. C. Díaz-Ramos, M. Domínguez-Vázquez, E. García-Río, J. M. Lorenzo-Naveiro and A. Rodríguez-Vázquez for the enlightening discussions in our daily life. I would also like to express my deep gratitude to Christoph Böhm, Ramiro Lafuente and Hiroshi Tamaru for their interesting comments and remarks on several aspects of this work.
\section{Preliminaries}\label{section:preliminaries}
This section is completely devoted to the introduction of the basic tools and terminology concerning symmetric spaces of non-compact type that will be utilized in the paper.

Let $M$ be a connected symmetric space of non-compact type and let $G$ be the connected component of the identity of the isometry group of $M$. Denote by $K$ the isotropy subgroup of $G$ at an arbitrary but fixed point $o \in M$. Then, $M$ turns out to be diffeomorphic to the coset $G/K$. Let $\g{g}$ be the real semisimple Lie algebra of $G$. Then $\g{k}$, the Lie algebra of $K$, is a maximal compact subalgebra of $\g{g}$. Let $B$ be the Killing form of $\g{g}$ and let us denote by $\g{p}$ the orthogonal complement of $\g{k}$ in $\g{g}$ with respect to $B$. Then $\g{g} = \g{k} \oplus \g{p}$ is a Cartan decomposition of the Lie algebra $\g{g}$. Let $\theta$ be the corresponding Cartan involution. Note that $\theta_{\rvert \g{k}} = \id_{\g{k}}$ and $\theta_{\rvert \g{p}} = -\id_{\g{p}}$. Moreover, we have that the expression $\langle X, Y \rangle_{B_{\theta}} = -B(X, \theta Y)$ defines a $\theta$-invariant positive definite inner product on $\g{g}$. This inner product satisfies 
\begin{equation}\label{eq:cartan:inner}
	\langle \ad (X)Y, Z \rangle_{B_{\theta}}=- \langle Y, \ad(\theta X)Z \rangle_{B_{\theta}},\qquad \text{for all }X,Y,Z\in\g{g}.
\end{equation}
Let $\g{a}$ be a maximal abelian subspace of $\g{p}$. Note that the rank of the symmetric space $M$ coincides with the dimension of $\g{a}$. Now, $\{\ad(H):H\in\g{a}\}$ is a commuting family of self-adjoint endomorphisms of $\g{g}$. This means that they diagonalize simultaneously. Indeed, define $\g{g}_\lambda:=\{X\in\g{g}: [H,X]=\lambda(H)X \text{ for all } H\in\g{a}\}$ for any $\lambda$ in $\g{a}^*$. Then, we obtain the restricted root space decomposition of $\g{g}$, which is the $\langle \cdot, \cdot \rangle_{B_{\theta}}$-orthogonal decomposition
\begin{equation*}\label{eq:root_space_decomposition}
\g{g}=\g{g}_0\oplus\biggl(\bigoplus_{\lambda\in\Sigma}\g{g}_\lambda\biggr),
\end{equation*}
where $\Sigma=\{\lambda\in\g{a}^*:\lambda\neq 0,\,\g{g}_\lambda\neq 0\}$ is the so-called set of restricted roots. Moreover, we have the relations
\begin{equation}\label{equation:bracket:relation}
[\g{g}_{\lambda},\g{g}_{\mu}] \subseteq \g{g}_{\lambda+\mu}
\end{equation}
and $\theta \g{g}_\lambda=\g{g}_{-\lambda}$, for any $\lambda$, $\mu\in\g{a}^*$. Define the vector $H_\lambda\in\g{a}$ by the relation $B(H_\lambda,H)=\lambda(H)$, for each $H\in\g{a}$. This allows to introduce an inner product in $\Sigma$ given by $\langle \alpha, \lambda \rangle = \langle H_{\alpha}, H_{\lambda} \rangle_{B_{\theta}}$, for each $\alpha$, $\lambda \in \Sigma$. This converts $\Sigma$ into an abstract root system in $\g{a}^*$. Hence, we can introduce a criterion of positivity in $\Sigma$. Let $\Pi$ be a set of simple roots for $\Sigma$ and denote by $\Sigma^+$ the resulting set of positive roots. Note that $\Sigma=\Sigma^+\cup(-\Sigma^+)$. Any root $\lambda \in \Sigma$ can be written as $\lambda = \sum_{\alpha \in \Pi }a_{\alpha} \alpha$, where the coefficients $a_{\alpha}$ are integers either all non-negative (if $\lambda \in \Sigma^+$) or all non-positive  (if $\lambda \in -\Sigma^+$). When $\lambda \in \Sigma^+$, the sum $l(\lambda) = \sum_{\alpha \in \Pi} a_{\alpha}$ is called the level of the root $\lambda$. We will make extensive use in this paper of the so-called Cartan integers, that is, the integers of the form $A_{\alpha, \lambda} := 2 \langle \alpha, \lambda \rangle|\alpha|^{-2}$, where $\alpha$, $\lambda \in \Sigma$. See~\cite{K} for more details on root systems.

Let us focus now on the Riemannian structure of the symmetric space of non-compact type $M$. Define $\g{n} = \bigoplus_{\lambda \in \Sigma^{+}} \g{g}_{\lambda}$, which is a nilpotent Lie subalgebra of $\g{g}$, as easily follows from the bracket relation~\eqref{equation:bracket:relation}. The direct sum decomposition $\g{g} = \g{k} \oplus \g{a} \oplus \g{n}$ is an Iwasawa decomposition of $\g{g}$. Note that $\g{a} \oplus \g{n}$ is a solvable Lie subalgebra of $\g{g}$. Let $A$, $N$ and $AN$ be the connected (closed) subgroups of $G$ with Lie algebras $\g{a}$, $\g{n}$ and $\g{a} \oplus \g{n}$ respectively. The isometry group $G$ is diffeomorphic to the product $KAN$. In addition, the action of $AN$ on $M$ is free and transitive, so we can equip $AN$ with a metric that converts $AN$ and $M$ into isometric Riemannian manifolds. This induced metric on $AN$ turns out to be left invariant. We will denote by $\langle\cdot,\cdot\rangle$ this left invariant Riemannian metric on $AN$, and also the corresponding inner product on $\g{a}\oplus\g{n}$.

Let us assume that $M$ is irreducible and let $X$, $Y$ be vectors in $\g{a}\oplus\g{n}$. Then, up to rescaling the metric on $M$, we have the relation between $\langle\cdot,\cdot\rangle$ and  $\langle\cdot,\cdot\rangle_{B_\theta}$ given by the expression
\begin{equation}\label{eq:relation:inner}
\langle X,Y\rangle =\langle X_\mathfrak{a},Y_\mathfrak{a}\rangle_{B_\theta}+\frac{1}{2}\langle X_\mathfrak{n},Y_\mathfrak{n}\rangle_{B_\theta},
\end{equation}
where $(\cdot)_{\g{a}}$ and $(\cdot)_{\g{n}}$ denote orthogonal projection onto $\g{a}$ and $\g{n}$ respectively, with respect to $\langle\cdot,\cdot\rangle_{B_\theta}$. Using Koszul formula and relations~\eqref{eq:relation:inner} and~\eqref{eq:cartan:inner}, we get that the Levi-Civita connection $\bar{\nabla}$ of the Lie group $AN$ reads as
\begin{equation}\label{eq:inner:an:bphi}
\langle \bar{\nabla}_X Y,Z\rangle =\frac{1}{4}\langle [X,Y]+[\theta X,Y]-[X,\theta Y],Z\rangle_{B_\theta},
\end{equation}
for any $X$, $Y$, $Z$ in $\g{a}\oplus\g{n}$. Note that we are making use of two different inner products in the previous formula.

We finish this section with some lemmas that will be useful in what follows. We also note that, from now on, whenever we consider a unit vector $X\in\g{a}\oplus\g{n}$, we will understand that it has length one with respect to the inner product $\langle\cdot,\cdot\rangle$ defined on $\g{a}\oplus\g{n}$.
\begin{lemma}\label{lemma:berndt:sanmartin} \emph{\cite[Lemma 2.3]{BS18}}
Let $\lambda \in \Sigma^{+}$ and $X \in \g{g}_{\lambda}$. Then: 
\begin{enumerate}[{\rm (i)}]
\item $[\theta X, X] = 2\langle X, X \rangle H_{\lambda}=\langle X, X \rangle_{B_{\theta}} H_{\lambda}$. \label{lemma:berndt:sanmartin:i}
\item $[\theta X, Y] \in \g{k}_0 = \g{g}_0 \ominus \g{a}$, for any $Y \in \g{g}_\lambda$ orthogonal to $X$. \label{lemma:berndt:sanmartin:ii}
\end{enumerate}
\end{lemma}

\begin{lemma}\label{lemma:bracket:root:spaces}
Let $\alpha$, $\beta$, $\alpha + \beta \in \Sigma$ be roots. Then:
\begin{enumerate}[{\rm (i)}]
\item $\spann \{ [X_\alpha, Y]  :  Y \in \g{g}_{\beta} \}$ is a non-zero vector subspace of $\g{g}_{\alpha + \beta}$, for any non-zero vector $X_{\alpha} \in \g{g}_{\alpha}$. \label{lemma:bracket:root:spaces:1}
\item $[\g{g}_{\alpha}, \g{g}_{\beta}] = \g{g}_{\alpha+\beta}$.	\label{lemma:bracket:root:spaces:2}
\end{enumerate}	
\end{lemma}
\begin{proof}
Assertion~(\ref{lemma:bracket:root:spaces:1}) corresponds to~\cite[Lemma~7.75]{K}. In order to prove claim~(\ref{lemma:bracket:root:spaces:2}), we will assume that there exists a non-zero vector $X \in \g{g}_{\alpha+\beta} \ominus [\g{g}_\alpha, \g{g}_\beta]$, where $\ominus$ stands for the orthogonal complement of $[\g{g}_\alpha, \g{g}_\beta]$ in $\g{g}_{\alpha+\beta}$, and show that this leads us to a contradiction. Take $Y \in  \g{g}_{-\beta}$ such that $[X, Y] \neq 0$, which is possible by virtue of~(\ref{lemma:bracket:root:spaces:1}). Now, using~\eqref{eq:cartan:inner} we have  
\begin{align*}
0 < \langle [X, Y],  [X, Y] \rangle_{B_{\theta}} = \langle X, [\theta Y,[X, Y]]  \rangle_{B_{\theta}}, 
\end{align*}
which is a contradiction since the right hand side term must be zero taking into account the choice of $X$ and the fact that $[\theta Y,[X, Y]] \in [\g{g}_{\beta}, \g{g}_{\alpha}]$. \qedhere
\end{proof}

\begin{lemma}\label{lemma:theta:alpha:new}
Let $\alpha$, $\lambda \in \Sigma^{+}$ be distinct roots such that  $\lambda-\alpha$ is not a root. Let $X_\alpha$ be a vector in $\g{g}_\alpha$ and let $X_\lambda$, $Y_\lambda$ be vectors in $\g{g}_\lambda$ Then:
\begin{enumerate}[{\rm (i)}]
\item $[[X_\alpha, X_\lambda], \theta X_\alpha] = - |\alpha|^2 A_{\alpha, \lambda} \langle X_\alpha, X_\alpha \rangle X_\lambda$. \label{lemma:theta:alpha:new:1}
\item $[[X_\alpha, X_\lambda], \theta X_\lambda] =  |\alpha|^2 A_{\alpha, \lambda} \langle X_\lambda, X_\lambda \rangle X_\alpha$. \label{lemma:theta:alpha:new:2}
\item $\langle [X_\alpha, X_\lambda], [X_\alpha, Y_\lambda] \rangle = -  |\alpha|^2 A_{\alpha, \lambda} \langle X_\alpha, X_\alpha \rangle \langle X_\lambda, Y_\lambda \rangle$.\label{lemma:theta:alpha:new:3}
\end{enumerate}
\end{lemma}

\begin{proof}
	Assertion~(\ref{lemma:theta:alpha:new:1}) follows from~\cite[Lemma 2.4 (ii)]{BS18} and assertion~(\ref{lemma:theta:alpha:new:2}) follows from~(\ref{lemma:theta:alpha:new:1}) by interchanging the roles of $\alpha$ and $\lambda$ (since $\lambda - \alpha$ is not a root by assumption, neither is $\alpha-\lambda)$ and the fact that $|\alpha|^2 A_{\alpha, \lambda} = |\lambda|^2 A_{\lambda, \alpha}$. Now, using~\eqref{eq:relation:inner}, ~\eqref{eq:cartan:inner} and assertion~(\ref{lemma:theta:alpha:new:1}), we get 
	\begin{align*}
		\langle [X_\alpha, X_\lambda], [X_\alpha, Y_\lambda] \rangle & = \frac{1}{2} \langle [X_\alpha, X_\lambda], [X_\alpha, Y_\lambda] \rangle_{B_\theta} = - \frac{1}{2} \langle  X_\lambda, [\theta X_\alpha, [X_\alpha, Y_\lambda]] \rangle_{B_\theta}\\
		& =  \frac{1}{2} \langle  X_\lambda, [[X_\alpha, Y_\lambda], \theta X_\alpha] \rangle_{B_\theta} =-\frac{1}{2} |\alpha|^2 A_{\alpha, \lambda}\langle X_\alpha, X_\alpha \rangle  \langle  X_\lambda, Y_\lambda \rangle_{B_\theta} \\
		& =  - |\alpha|^2 A_{\alpha, \lambda} \langle X_\alpha, X_\alpha \rangle \langle  X_\lambda, Y_\lambda \rangle,
	\end{align*}
	for any $X_\alpha$ in $\g{g}_\alpha$ and any $X_\lambda$, $Y_\lambda$ in $\g{g}_\lambda$. This proves claim~(\ref{lemma:theta:alpha:new:3}). 
\end{proof}

\begin{lemma}\label{lemma:root:spaces:k0}
Let $\alpha \in \Sigma^{+}$ and let $X \in \g{g}_{\alpha}$. Then $[T, X] \in \g{g}_\alpha \ominus \R X$ for any $T \in \g{k}_0 = \g{g}_0 \ominus \g{a}$.
\end{lemma}
\begin{proof}
Since $T$ is in $\g{k}_0 \subset \g{g}_0$, we have that $[T, X]$ is in $\g{g}_\alpha$ by virtue of~\eqref{equation:bracket:relation}. Moreover, using~\eqref{eq:cartan:inner} and $\theta_{\rvert_{\g{k}}} = \id_{\g{k}}$, we get that $\langle [T, X], X \rangle_{B_\theta} = -\langle X, [\theta T, X] \rangle_{B_\theta}  = -\langle X, [T, X] \rangle_{B_\theta}$, which proves this result. 
\end{proof}
\section{Nilsolitons as Lie subgroups of nilpotent Iwasawa groups}\label{section:nilsolitons}
In this section, we justify that any connected expanding homogeneous Ricci soliton can be obtained, up to isometry, as a completely solvable extension of a Lie subgroup of a nilpotent Iwasawa group that is a Ricci soliton (equivalently, a nilsoliton) when considered with the induced metric. More precisely, this section is devoted to proving the following

\begin{claim*}
Any connected expanding homogeneous Ricci soliton is isometric to a simply connected solvsoliton that can be constructed, up to isometry, as a completely solvable extension of a nilsoliton. This nilsoliton is isometric and isomorphic to a Lie subgroup of some nilpotent Iwasawa group when considered with the induced metric.
\end{claim*}

As a consequence of the Alekseevskii conjecture, recently proved to be true in~\cite{BoLa21}, one can state: any connected homogeneous Einstein manifold with negative scalar curvature is isometric to a simply connected Einstein solvmanifold. This happens to be equivalent to the following (a priori stronger) statement, usually known as generalized Alekseevskii conjecture~\cite[Theorem~3]{Ja14}: every connected expanding homogeneous Ricci soliton is isometric to a simply connected solvmanifold. Combined with~\cite[Theorem~1.1]{Ja15}, it allows to reduce the investigation of connected expanding homogeneous Ricci solitons to the study of simply  connected non-flat solvsolitons. Such solvsolitons must have $c < 0$ in the defining equation of algebraic Ricci solitons~\eqref{equation:nilsoliton}, as follows from combining~\cite[Proposition~4.6]{La11C} with~\cite[Theorem~1]{AlKi}. Note also that non-flat solvsolitons are diffeomorphic to a Euclidean space and hence simply connected~\cite[Theorem~1.1]{Ja15}.

Now, any non-flat solvsoliton (in particular any non-flat Einstein solvmanifold) can be constructed as a completely solvable extension of a nilsoliton. Let $L$ be a non-flat nilsoliton with Lie algebra $\g{l}$. Since we will deal with left invariant metrics on Lie groups, it will suffice to work with their Lie algebras endowed with the induced inner product. Then, the semidirect product of $\g{l}$ with any abelian Lie algebra $\g{b}$ of symmetric derivations of~$\g{l}$, endowed with a suitable metric, happens to be a non-flat solvsoliton~\cite[Proposition~4.3]{La11C}. This extension is completely solvable, as we justify in the next paragraph. Moreover, any non-flat solvsoliton can be constructed, up to isometry, following this procedure~\cite[Corollary~4.10]{La11C}. This fact converts the study of solvsolitons (in particular of Einstein solvmanifolds) into a problem of nilsolitons.

Let us justify that $\g{b} \oplus \g{l}$ is completely solvable. On the one hand, since $\g{b}$ is abelian and any derivation of $\g{l}$ in $\g{b}$ is symmetric, then the endomorphism $\ad(B)$ of $\g{b} \oplus \g{l}$, for any $B \in \g{b}$, diagonalizes with real eigenvalues. On the other hand, for a suitable basis of $\g{n}$, all the endomorphims $\ad(X)$ of $\g{n}$, with $X \in \g{n}$, can be regarded as upper triangular matrices with zeros in the diagonal, as follows from~\cite[Theorem~1.35]{K}. Hence, since $[\g{b} \oplus \g{l}, \g{b} \oplus \g{l} ] \subset \g{l}$, then $\ad(X)$ can be also identified with an upper triangular matrix with zeros in the diagonal when viewed as an endomorphism of $\g{b} \oplus \g{n}$, for any $X \in \g{n}$. Thus, the endomorphism $\ad(X)$ of $\g{b} \oplus \g{l}$ has real eigenvalues for any $X \in \g{n}$. This proves that $\g{b} \oplus \g{l}$ is a completely solvable Lie algebra. 

In summary, we have proved so far that any connected expanding homogeneous Ricci soliton is isometric to a simply connected non-flat solvsoliton that can be obtained, up to isometry, as a completely solvable extension of a non-flat nilsoliton. In order to finish the proof of the Claim, we will justify that any non-flat nilsoliton is isometric and isomorphic to a Lie subgroup of a nilpotent Iwasawa group when considered with the induced metric.

Let $L$ be a nilsoliton with Lie algebra $\g{l}$ and whose $(1,1)$-Ricci tensor, $\Ric^L$, reads as $\Ric^L = c \id + D$, for some $c < 0$ and a symmetric derivation $D$ of $\g{l}$. We can and will assume that $D$ is a non-trivial derivation. Otherwise, $\g{l}$ would be an abelian Lie algebra~\cite[Theorem~ 2.4]{Mi76} and hence $L$ would be flat, which is a contradiction. Thus, by virtue of~\cite[Proposition~4.3]{La11C}, the semidirect product $\g{s} = \mathbb{R} D \oplus \g{l}$ endowed with a suitable metric is completely solvable and Einstein with negative Ricci curvature. Note that $\g{l}$ is the nilradical (maximal nilpotent ideal) of $\g{s}$, since $D$ is a non-zero symmetric derivation of $\g{l}$. Let $S$ be the corresponding simply connected Einstein solvmanifold with Lie algebra $\g{s}$. Hence $S$ is standard~\cite{La10}, that is, the orthogonal complement of $[\g{s}, \g{s}]$ in $\g{s}$ is abelian. Moreover, $\g{s}$ is non-unimodular ($\tr \ad (X) \neq 0$ for some $X \in \g{s}$), since otherwise $S$ would be flat~\cite[Proposition~4.9]{He98}. Thus, by virtue of~\cite[Corollary~4.11]{He98}, we get that the nilradical and the derived Lie algebra of $\g{s}$ coincide, that is, $\g{l} = [\g{s}, \g{s}]$.

Now, we will make use of the following remarkable result achieved by Jablonski~\cite[Theorem 0.2]{Jab:arxiv}, lying at the intersection of Ado's theorem and Nash embedding theorem, which guarantees that we can construct an isometric, isomorphic embedding from a completely solvable Lie group endowed with a left invariant Einstein metric into some irreducible symmetric space of non-compact type. 

Let $L$ and $S$ be the nilsoliton and the Einstein completely solvable extension of $L$ described above, respectively. Hence, using Jablonski's result, there exists an isometric, isomorphic embedding $\varphi \colon S \to AN$, where $AN$ is the solvable Lie group model for an irreducible symmetric space of non-compact type (see Section~\ref{section:preliminaries} for further details). Recall that $\g{a} \oplus \g{n}$ denotes the solvable Lie algebra of $AN$, and note that $[\g{a} \oplus \g{n}, \g{a} \oplus \g{n}] = \g{n}$. Now, using the fact that $\g{l} = [\g{s}, \g{s}]$ proved above, and that the Lie exponential map is a diffeomorphism for simply connected nilpotent groups~\cite[Theorem~1.127]{K}, we get 
\[
\varphi (L) = \varphi \circ \Exp (\g{l}) = \Exp \circ \, d\varphi [\g{s}, \g{s}] \subset \Exp ([\g{a} \oplus \g{n}, \g{a} \oplus \g{n} ]) = \Exp (\g{n}) = N,
\]  
where $\Exp$ denotes the Lie exponential map (for both $L$ and $N$) and $d \varphi$ denotes the map between the Lie algebras $\g{s}$ and $\g{a} \oplus \g{n}$ induced by $\varphi$. This means that the initial non-flat nilsoliton $L$ is isometric and isomorphic to a Lie subgroup $\varphi(L)$ of some nilpotent Iwasawa group $N$. This completes the proof of the Claim.
\section{General setting} \label{section:general:setting}
Let $N$ be a nilpotent Iwasawa group, that is, the nilpotent group of the Iwasawa decomposition of the connected component of the identity of the isometry group of an irreducible symmetric space of non-compact type $M \cong AN$. Let $S$ be a codimension one Lie subgroup of $N$. Recall that the main aim of this work is to determine under which circumstances $S$ is itself, with the induced metric, a Ricci soliton (equivalently, a nilsoliton). This section is devoted to establish the general setting of the paper. Let us be more precise.

Since we will be dealing with $S$, $N$ and $AN$, which are Lie groups endowed with a left invariant metric, their tangent spaces are spanned by the left invariant vector fields of their corresponding Lie algebras $\g{s}$, $\g{n}$ and  $\g{a} \oplus \g{n}$, respectively. All in all, since $S$ and $N$ will be treated as homogeneous submanifolds of $M$, it will suffice to study their geometry at the tangent space at the origin, which will be identified with their corresponding Lie algebras.

Hence, in the first part of this section, we give a general description of the Lie subalgebra $\g{s}$ of $\g{n}$ of a codimension one Lie subgroup $S$ of $N$, by means of a general expression of a global left invariant unit normal vector field $\xi$ to $S$ in $N$. After that, we deduce an equation for the Ricci tensor of $S$ by investigating the geometry of $S$ and $N$ as Lie subgroups of $N$ and $AN$, respectively.

In the last part of this section, we define an endomorphism $D$ of $\g{s}$ which depends on a real number $c$, and we justify that $S$ is a Ricci soliton Lie subgroup of $N$ if and only if $D$ is a derivation for some value of $c$.

Recall that we denote by $\g{s}$ and $\g{n}$ the Lie algebras of $S$ and $N$, respectively. Hence, we can write $\g{s} = \g{n} \ominus \R \xi$ (orthogonal complement of $\R \xi$ in $\g{s}$) for some $\xi \in \g{n}$. However, $\g{s}$ is not always a subalgebra for all choices of $\xi \in \g{n}$. We make this idea more precise in the following

\begin{lemma}\label{lemma:subalgebra:s:xi}
Let $\g{s} = \g{n} \ominus \R \xi$ be a subalgebra of $\g{n}$ for some $\xi \in \g{n}$. Then, we have that 
\begin{equation}\label{equation:xi}
\xi = \sum_{\gamma \in \Phi} a_{\gamma} \xi_{\gamma},
\end{equation}
where $\xi_{\gamma}$ is a unit vector of $\g{g}_{\gamma}$ and $a_{\gamma}$ is a positive number, for each $\gamma \in \Phi$ and a certain subset $\Phi$ of $\Pi$.
\end{lemma}

\begin{proof}
Since $\xi$ is in $\g{n} = \bigoplus_{\gamma \in \Sigma^{+}} \g{g}_\gamma$, we have that $\xi = \sum_{\gamma \in \Sigma^{+}} X_{\gamma}$, with $X_{\gamma}$ in $\g{g}_{\gamma}$ for each $\gamma \in \Sigma^{+}$. If we prove that $X_{\gamma} = 0$ for all $\gamma \in \Sigma^{+} \backslash \Pi$, then the result will follow.

Let us write $\lambda$ for a highest level root in $\Sigma^{+}$ such that $X_{\lambda} \neq 0$. We will assume that $l(\lambda) > 1$ (otherwise the result follows) and we will see that this would imply that $X_{\lambda}$ is in $\g{s}= \g{n} \ominus \R \xi$, which is a contradiction, since by assumption $X_\lambda$ cannot be orthogonal to the unit normal vector $\xi$.

Note that the choice of $\lambda$ implies that $X_{\nu} = 0$ and therefore $\g{g}_{\nu} \subset \g{s}$ for each $\nu \in \Sigma^{+}$ satisfying $l(\nu) > l(\lambda)$. Since we are assuming that $\lambda$ is not a simple root, we can write $\lambda = \gamma_1 + \gamma_2$ for some positive roots $\gamma_1$, $\gamma_2 \in \Sigma^{+}$. By virtue of Lemma~\ref{lemma:bracket:root:spaces}~(\ref{lemma:bracket:root:spaces:2}) we have 
\[
X_\lambda = \sum_{i=1}^n [Y^i_1, Y^i_2] 
\]
for some positive integer $n$ and $Y_k^i \in \g{g}_{\gamma_k}$, with $i \in \{1, \dots , n\}$ and $k \in \{1,2\}$. Note that the elements $Y_k^i$ do not necessarily belong to $\g{s}$. Indeed, consider the orthogonal decomposition $\g{g}_{\gamma_k}=(\g{g}_{\gamma_k} \ominus \R X_{\gamma_k}) \oplus \R X_{\gamma_k}$, for each $k \in \{1,2\}$. Then, we will write $Y^i_k = Z_k^i + b_k^i X_{\gamma_k}$, where $Z_k^i$ is in $(\g{g}_{\gamma_k} \ominus \R X_{\gamma_k}) \subset \g{s}$  and  $b_k^i$ is a real number, with $i \in \{1, \dots , n\}$ and $k \in \{1,2\}$. The vector $\eta_k^i = b_k^i |X_\lambda|^{-2} (|X_\lambda|^2 X_{\gamma_k}-|X_{\gamma_k}|^2 X_\lambda)$ belongs to $\g{s}$ for any $i \in \{1, \dots , n\}$ and any $k \in \{1,2\}$, since it is orthogonal to $\xi$. Hence, combining the equality
\begin{align*}
X_{\lambda} & = \sum_{i=1}^n [Y^i_1, Y^i_2]= \sum_{i=1}^n [Z^i_1 + \eta^i_1 +b^i_1 |X_{\gamma_1}|^2 |X_\lambda|^{-2} X_\lambda, Z^i_2 + \eta^i_2 +b^i_2 |X_{\gamma_2}|^2 |X_\lambda|^{-2} X_\lambda]\\
& = \sum_{i=1}^n [Z^i_1 + \eta^i_1, Z^i_2 + \eta^i_2] + [Z^i_1 + \eta^i_1, b^i_2 |X_{\gamma_2}|^2 |X_\lambda|^{-2} X_\lambda ]
 + [b^i_1 |X_{\gamma_1}|^2 |X_\lambda|^{-2} X_\lambda, Z^i_2 + \eta^i_2] 
\end{align*}
with~\eqref{equation:bracket:relation}, and recalling that $\g{g}_{\nu} \subset \g{s}$ for any $\nu \in \Sigma^{+}$ such that $l(\nu) > l(\lambda)$, we get that $X_\lambda \in \g{s} \oplus \g{g}_{\gamma_1 + \lambda} \oplus  \g{g}_{\gamma_2 + \lambda} \subset \g{s}$.
\end{proof}

According to the above result, the choice of a codimension one Lie subgroup $S$ of $N$ determines a subset $\Phi$ of the set of simple roots. From now on we will assume that $\g{s} = \g{n} \ominus \R \xi$ with $\xi$ as in~\eqref{equation:xi}.

We will investigate now $\Ric^N$, the Ricci tensor of $N$. Let $\g{a} \oplus \g{n}$ be the Lie algebra of $AN$ and let $\{ H_i\}_{i=1}^n$ be an orthonormal basis of $\g{a}$. Let us denote by $\bar{\nabla}$ the Levi-Civita connection of $AN$. Let us write $\bar{\Ss}_{H}$ and $\bar{R}_{H}$ for the shape operator and the Jacobi operator of $N$ as a submanifold of $AN$ with respect to a unit normal vector $H \in \g{a}$, respectively. From Gauss equation, we can derive the following expression for $\Ric^N$, the~$(1,1)$-Ricci tensor of $N$:   
\begin{equation}\label{equation:ricN}
\Ric^N = (\Ric^{AN}_{\rvert_{\g{n}}} )^\top + \bar{\Ss}_{\mathcal{H}} - \sum_{i=1}^n (\bar{R}^\top_{H_i} + \bar{\Ss}^2_{H_i}),
\end{equation} 
where $\Ric^{AN}$ denotes the $(1,1)$-Ricci tensor of $AN$, $(\cdot)^\top$ stands for the orthogonal projection onto $\g{n}$ and $\mathcal{H}$ is the mean curvature vector of $N$ in $AN$. We state a lemma which will lead us to a simplified version of~\eqref{equation:ricN}.

\begin{lemma}\label{lemma:ricN}
We have:
\begin{enumerate}[{\rm (i)}]
\item $(\bar{R}^\top_H + \bar{\Ss}^2_{H}) X = 0$ for all $H \in \g{a}$ and $X \in \g{n}$.\label{lemma:ricN:3}
\item $\bar{\Ss}_\mathcal{H} X = \ad(\mathcal{H}) X$ for all $X \in \g{n}$, with $\mathcal{H} = \sum_{\lambda \in \Sigma^{+}} \dim \g{g}_{\lambda} H_{\lambda}$.\label{lemma:ricN:4}
\end{enumerate}
\end{lemma}

\begin{proof}
First, using~\eqref{eq:inner:an:bphi} we deduce 
\begin{align}\label{lemma:ricN:eq1}
\langle \bar{\nabla}_{H} X, Z \rangle & = \frac{1}{4} \langle [H, X] + [\theta H, X] - [H, \theta X]  , Z \rangle_{B_{\theta}} = 0, 
\end{align}
for all $X$, $Z \in \g{a} \oplus \g{n}$ and all $H \in \g{a}$. Now, let $X \in \g{n}$ and $H \in \g{a}$. On the one hand, using~(\ref{eq:inner:an:bphi}) and~(\ref{eq:relation:inner}) taking into account that $[\g{a}, \g{n}] \subset \g{n}$, we get
\begin{align*}
\langle \bar{\nabla}_{X} H, Z \rangle & = \frac{1}{4} \langle [X, H] + [\theta X, H] - [X, \theta H]  , Z \rangle_{B_{\theta}}= -\frac{1}{2} \langle \ad(H) X, Z \rangle_{B_\theta}  = -\langle \ad(H) X, Z \rangle,
\end{align*}
for all $Z \in \g{a} \oplus \g{n}$. Using this we get  
\begin{equation}\label{lemma:ricN:eq3}
\bar{\Ss}_H X = (- \bar{\nabla}_X H)^\top = (\ad(H) X)^\top = \ad (H)X,
\end{equation}
for all $X \in \g{n}$ and all $H \in \g{a}$, where $(\cdot)^\top$ stands for the orthogonal projection onto $\g{n}$. On the other hand, using~\eqref{lemma:ricN:eq1} twice and~\eqref{lemma:ricN:eq3}, we obtain 
\begin{align*}
\bar{R}^\top_H  X & = \left( \bar{\nabla}_X \bar{\nabla}_H H - \bar{\nabla}_H \bar{\nabla}_X H + \bar{\nabla}_{\ad(H)X} H \right)^\top  = \left( \bar{\nabla}_{\bar{\Ss}_H X} H \right)^\top =- \bar{\Ss}_H (\bar{\Ss}_H X) =- \bar{\Ss}^2_H  X,
\end{align*}
for all $X \in \g{n}$ and all $H \in \g{a}$, and hence assertion~(\ref{lemma:ricN:3}) follows. Finally, let us calculate the mean curvature vector $\mathcal{H}$ of $N$ as a submanifold of $AN$. Let $X_{\lambda}$ be a unit vector in $\g{g}_{\lambda}$, for any $\lambda \in \Sigma^{+}$. First, using~\eqref{lemma:ricN:eq3} we have  
\begin{equation}\label{eq:second:fund}
\langle \II(X_{\lambda}, X_{\lambda}), H \rangle  = \langle \bar{\Ss}_H X_{\lambda}, X_{\lambda} \rangle = \langle \ad(H) X_{\lambda}, X_{\lambda} \rangle = \langle H, H_{\lambda} \rangle,
\end{equation}
for any $H \in \g{a}$. Recall that $\{ H_i\}_{i=1}^n$ is an orthonormal basis of $\g{a}$. Using~\eqref{eq:second:fund} we deduce 
\begin{align*}
\mathcal{H} & = \sum_{\lambda \in \Sigma^{+}} \dim \g{g}_{\lambda} \II(X_{\lambda}, X_{\lambda}) =  \sum_{\lambda \in \Sigma^{+}}  \dim \g{g}_{\lambda} \sum_{i=1}^n  \langle \II(X_{\lambda}, X_{\lambda}), H_i \rangle H_i \\
& = \sum_{\lambda \in \Sigma^{+}}  \dim \g{g}_{\lambda} \sum_{i=1}^n  \langle H_i, H_{\lambda} \rangle H_i =  \sum_{\lambda \in \Sigma^{+}}\dim \g{g}_{\lambda} H_{\lambda},
\end{align*}
which combined with~\eqref{lemma:ricN:eq3}, noting that $\mathcal{H}$ is in $\g{a}$, leads us to assertion~(\ref{lemma:ricN:4}).
\end{proof}
Now, taking Lemma~\ref{lemma:ricN}~(\ref{lemma:ricN:3})-(\ref{lemma:ricN:4}) into account, together with the fact that an irreducible symmetric space of non-compact type is Einstein with negative scalar curvature, we can rewrite~\eqref{equation:ricN} as $\Ric^N = k \id + \ad(\mathcal{H})$, for some $k < 0$ and $\mathcal{H} = \sum_{\lambda \in \Sigma^{+}} \dim \g{g}_{\lambda} H_{\lambda}$. It is clear that $\ad(\mathcal{H})$ is a derivation of $\g{n}$, which proves in particular that $N$ is a Ricci soliton. 

Let us focus on the calculation of the Ricci tensor $\Ric$ of $S$. Let us denote by $\Ss_\xi$ and $R_{\xi}$ the shape operator and the Jacobi operator of $S$ with respect to $\xi$ as a codimension one Lie subgroup of $N$, respectively. From Gauss equation we deduce that
\begin{align} \label{equation:ric:1}
\Ric &= (\Ric^{N}_{\rvert_{\g{s}}} )^\top + \tr(\Ss_{\xi}) \Ss_{\xi} - R_{\xi} - \Ss^2_{\xi} \nonumber \\
& = k \id + (\ad(\mathcal{H})_{\rvert_{\g{s}}})^\top + \tr(\Ss_{\xi}) \Ss_{\xi} - R_{\xi} - \Ss^2_{\xi},
\end{align}
where $(\cdot)^\top$ stands now for the orthogonal projection onto $\g{s}$. It is interesting to mention that $\ad(\mathcal{H})$ does not need to be an endomorphism of $\g{s}$. This will be one of the main difficulties in order to calculate explicitly the $(1,1)$-Ricci tensor of $S$ in Section~\ref{section:ricci:operator}. Now, for each pair of roots $\alpha$, $\lambda \in \Phi$, we define the coefficient 
\[
l_{\alpha, \lambda} = (a^2_\alpha + a^2_\lambda)^{-1/2}
\]
and the following vector in $\g{s}$:
\begin{equation}\label{equation:definitio:eta}
\eta_{\alpha, \lambda} =   l_{\alpha, \lambda} (a_\lambda \xi_\alpha - a_\alpha \xi_\lambda),
\end{equation}
which is a unit vector when $\alpha$ is distinct from $\lambda$, and zero otherwise. For an arbitrary but fixed root $\alpha \in \Phi$, we have the decomposition
\begin{equation}\label{equation:tangent:decomposition}
\g{s}=\left( \bigoplus_{\gamma\in\Sigma^{+} \backslash \Phi } \g{g}_{\gamma} \right) \oplus \left( \bigoplus_{\gamma \in \Phi} (\g{g}_{\gamma} \ominus \R \xi_{\gamma}) \right) \oplus \left( \bigoplus_{\lambda \in \Phi \backslash \{ \alpha \} } \R \eta_{\alpha, \lambda} \right) 
\end{equation}
of the tangent space to $S$. Note that the elements in the third addend of~\eqref{equation:tangent:decomposition} are not mutually orthogonal. We state a result which will simplify~\eqref{equation:ric:1}.
\begin{lemma}\label{lemma:simplify:ric:S}
Let $X$ be a vector in $\g{s}$. Then, we have:
\begin{enumerate}[{\rm (i)}]
\item $[X, \xi] \in \g{s}$ and $[[X, \xi], \theta \xi] \in \g{n}$.\label{lemma:simplify:ric:S:1} 
\item $\Ss_{\xi} X = (- 1/2) [X, \xi] + (1/2) [X, \theta \xi]_\g{n}$, where $(\cdot)_\g{n}$ stands for the orthogonal projection onto $\g{n}$. In particular, $[X, \theta \xi]_\g{n} \in \g{s}$.\label{lemma:simplify:ric:S:2}
\item $\tr (\Ss_{\xi}) = 0$, i.e. $S$ is minimal in $N$. \label{lemma:simplify:ric:S:3}
\end{enumerate}
\end{lemma}

\begin{proof}
Since $X$ is in $\g{s} \subset \g{n}$, we have $X \in \bigoplus_{\gamma \in \Sigma^{+}} \g{g}_{\gamma}$. Using~\eqref{equation:bracket:relation} we deduce 
\[
[X, \xi] \in \bigoplus_{\gamma \in \Sigma^{+}\backslash \Pi} \g{g}_{\gamma} \subset \g{s},
\]
and thus the first claim in assertion~(\ref{lemma:simplify:ric:S:1}) follows. Using the above expression, \eqref{equation:bracket:relation} and that $\xi$ has non-trivial projection only onto root spaces associated with the simple roots in $\Phi$, the second claim in~(\ref{lemma:simplify:ric:S:1}) easily follows. 

Let $\nu$ be a positive root and $\gamma$ a simple root. Then $-\nu + \gamma$ is never a positive root, and not even a root if $\nu$ were simple. We will use these considerations several times in this proof. Now, we deduce that $[\theta X_{\nu}, \xi]$ is orthogonal to $\g{n}$ for any $X_{\nu} \in \g{g}_{\nu} \cap \g{s}$. Using Lemma~\ref{lemma:berndt:sanmartin}~(\ref{lemma:berndt:sanmartin:i}), we get that $[\theta \eta_{\alpha, \lambda}, \xi ]$ belongs to $\g{a}$ for any pair of roots $\alpha$, $\lambda \in \Phi$. Altogether, we have that $[\theta X, \xi]$ is orthogonal to $\g{n}$ for all $X \in \g{s}$. Using this together with~\eqref{eq:inner:an:bphi}, assertion~(\ref{lemma:simplify:ric:S:1}) and~\eqref{eq:relation:inner}, we get 
\begin{align*}
\langle \bar{\nabla}_X \xi, Z \rangle & = \frac{1}{4} \langle [X, \xi] + [\theta X, \xi] - [X, \theta \xi], Z \rangle_{B_\theta} \\
& = \frac{1}{4} \langle [X, \xi]  - [X, \theta \xi]_\g{n}, Z_\g{n} \rangle_{B_\theta} + \frac{1}{4} \langle  [\theta X, \xi]_\g{a} - [X, \theta \xi]_\g{a}, Z_\g{a} \rangle_{B_\theta}\\
& = \frac{1}{2} \langle [X, \xi]  - [X, \theta \xi]_\g{n}, Z \rangle + \frac{1}{4} \langle  [\theta X, \xi]_\g{a} - [X, \theta \xi]_\g{a}, Z \rangle,
\end{align*}
for any $X \in \g{s}$, any $Z \in \g{a} \oplus \g{n}$, and where $(\cdot)_\g{n}$ and $(\cdot)_\g{a}$ stand for the orthogonal projections with respect to $\langle\cdot,\cdot\rangle_{B_\theta}$ onto $\g{n}$ and $\g{a}$, respectively. Let $\nabla$ be the Levi-Civita connection of $N$. Since $S$ has codimension one in $N$ we have  
\[
\Ss_\xi X = - \nabla_X \xi = - (\bar{\nabla}_X \xi)_\g{n} = - \frac{1}{2} [X, \xi] + \frac{1}{2} [X, \theta \xi]_\g{n},
\]
which proves assertion~(\ref{lemma:simplify:ric:S:2}). 

Finally, we will see that $\tr (\Ss_{\xi}) = 0$. Let $X$ be either an element in $\g{g}_\nu \cap \g{s}$ for some positive root $\nu$, or $X = \eta_{\alpha,\lambda}$ for some simple roots $\alpha$, $\lambda \in \Phi$. Then $[\theta X, X]$ is in $\g{a}$, as follows from Lemma~\ref{lemma:berndt:sanmartin}~(\ref{lemma:berndt:sanmartin:i}) and the fact that $\pm (\lambda-\alpha)$ are not roots. Using this, assertion~(\ref{lemma:simplify:ric:S:1}), \eqref{eq:cartan:inner} and that $\xi \in \g{n}$, we have  
\[
\langle [X, \xi], X \rangle = \frac{1}{2}  \langle [X, \xi], X \rangle_{B_\theta} = -\frac{1}{2}  \langle  \xi,[\theta  X, X] \rangle_{B_\theta} = 0.
\]
Similarly and using that $[X, \theta \xi] \in \g{g}_0 \oplus \g{n}$, we also deduce 
\[
\langle [X, \theta \xi]_\g{n}, X \rangle = \frac{1}{2}  \langle [X, \theta \xi]_\g{n}, X \rangle_{B_\theta} = \frac{1}{2} \langle [X, \theta \xi], X \rangle_{B_\theta} = -\frac{1}{2}  \langle \theta \xi,[\theta  X, X] \rangle_{B_\theta} = 0.
\]
According to decomposition~\eqref{equation:tangent:decomposition} of $\g{s}$ and assertion~(\ref{lemma:simplify:ric:S:2}), we have just proved that $\langle \Ss_{\xi} X, X \rangle = 0$ for any $X$ as above. Thus, we deduce that $\tr (\Ss_{\xi}) = 0$, which finishes the proof.
\end{proof}

\begin{remark}\label{remark:normal:subgroup}
With the same argument that we have utilized to prove the first assertion in Lemma~\ref{lemma:simplify:ric:S}~(\ref{lemma:simplify:ric:S:1}), one can see that $[\g{s}, \g{n}] \subset \g{s}$, and consequently we get that $S$ is a normal subgroup of $N$. 
\end{remark}

Combining~\eqref{equation:ric:1} and Lemma~\ref{lemma:simplify:ric:S}~(\ref{lemma:simplify:ric:S:3}), we get that the Ricci tensor $\Ric$ of $S$ reads as
\[
\Ric =  k \id + (\ad(\mathcal{H})_{\rvert_{\g{s}}})^\top - R_{\xi} - \Ss^2_{\xi},
\]
for some $k < 0$. Recall from Section~\ref{section:introduction} that $S$ is a Ricci soliton if and only if there exists a derivation $D$ of $\g{s}$ and a real number $c$ such that $\Ric = D + c \id$. Hence, $S$ is a Ricci soliton if and only if for a certain $c$ the endomorphism  
\begin{equation}\label{definition:D}
D = (\ad(\mathcal{H})_{\rvert_{\g{s}}})^\top - R_{\xi} - \Ss^2_{\xi} + c \id 
\end{equation}
of $\g{s}$ is also a derivation of $\g{s}$. Therefore, we need to determine the endomorphism $D$ and next section is completely devoted to this purpose.
\section{The Ricci operator of $S$}\label{section:ricci:operator}
Let $S$ be a codimension one Lie subgroup of a nilpotent Iwasawa group $N$. Recall that we are trying to understand in which circumstances $S$ is itself, with the induced metric, a Ricci soliton. According to the (last paragraph of the) previous section, this is equivalent to determining if the endomorphism $D$ of $\g{s}$ defined in~\eqref{definition:D} is a derivation for some real value $c$. 

Hence, the next step consists in the investigation of the endomorphism $D$. This is the main aim of this section. First, we explicitly compute the term $R_{\xi} + \Ss^2_{\xi}$ when applied to any $X \in \g{s}$ (see Propositions~\ref{proposition:rxi:sxi} and~\ref{propostion:rs:eta}), except for the case $X \in \g{g}_\alpha \ominus \R \xi_\alpha$, where $\alpha$ is in $\Phi$ and $2 \alpha$ is a root. This case is very difficult to tackle without further hypotheses on $\Phi$. Actually, we first reduce to $\Phi = \{ \alpha \}$ ($2\alpha \in \Sigma^+$) combining Corollary~\ref{corollary:summary} and Proposition~\ref{proposition:B2}, and we address this last case with an Ad hoc argument in Lemma~\ref{lemma:orthogonal} and Proposition~\ref{proposition:1:root}. Finally, in Proposition~\ref{proposition:ad}, we study the remaining non-trivial term of $D$, that is, $(\ad(\mathcal{H})_{\rvert_{\g{s}}})^\top$.

Recall that $\bar{\nabla}$ and $\nabla$ denote the Levi-Civita connection of $M \cong AN$ and $N$ respectively.

\begin{lemma}\label{lemma:levi:connection:xi}
We have:
\begin{multicols}{3}
\begin{enumerate}[{\rm (i)}]
\item $[\theta \xi, \xi ] = 2 \sum_{\gamma \in \Phi} a_{\gamma}^2 H_{\gamma}$,\label{lemma:levi:connection:xi:1}
\item $\bar{\nabla}_\xi \xi = \sum_{\gamma \in \Phi} a_{\gamma}^2 H_{\gamma}$,\label{lemma:levi:connection:xi:2}
\item $\nabla_\xi \xi = 0$.\label{lemma:levi:connection:xi:3}
\end{enumerate}
\end{multicols}
\end{lemma}

\begin{proof}
Using the fact that $\alpha-\beta$ is neither a root nor zero for any distinct $\alpha$, $\beta \in \Pi$, and Lemma~\ref{lemma:berndt:sanmartin}~(\ref{lemma:berndt:sanmartin:i}), we get 
\begin{align*}
[\theta \xi, \xi] &= \sum_{\gamma, \lambda \in \Phi} a_{\gamma} a_{\lambda} [\theta \xi_\gamma, \xi_\lambda] = \sum_{\gamma \in \Phi} a^2_{\gamma} [\theta \xi_\gamma, \xi_\gamma] = 2 \sum_{\gamma \in \Phi} a_{\gamma}^2 H_{\gamma},
\end{align*}
which proves assertion~(\ref{lemma:levi:connection:xi:1}). Moreover, using~\eqref{eq:inner:an:bphi}, assertion~(\ref{lemma:levi:connection:xi:1}), the fact that $[\theta \xi, \xi ]$ is in $\g{a}$ and~\eqref{eq:relation:inner}, we deduce 
\begin{align*}
\langle \bar{\nabla}_\xi \xi, Z \rangle & = \frac{1}{4} \langle [\xi, \xi] + [\theta \xi, \xi] - [\xi, \theta \xi], Z \rangle_{B_\theta} = \frac{1}{2} \langle [\theta \xi, \xi], Z \rangle_{B_\theta}  \\
& =  \langle \sum_{\gamma \in \Phi} a_{\gamma}^2 H_{\gamma}, Z \rangle_{B_\theta} = \langle \sum_{\gamma \in \Phi} a_{\gamma}^2 H_{\gamma}, Z \rangle,
\end{align*}
for any $Z \in \g{a} \oplus \g{n}$, which proves claim~(\ref{lemma:levi:connection:xi:2}). Finally, using the fact that $\bar{\nabla}_\xi \xi$ is in $\g{a}$, as follows from assertion~(\ref{lemma:levi:connection:xi:2}), we have $\nabla_\xi \xi = (\bar{\nabla}_\xi \xi )_\g{n} = 0$, which concludes the proof.
\end{proof}

First, we will calculate $R_{\xi} + \Ss^2_{\xi}$ when restricted to the first two addends of decomposition~\eqref{equation:tangent:decomposition} of the Lie algebra $\g{s}$ of $S$.
\begin{proposition}\label{proposition:rxi:sxi}
Let $S$ be a codimension one Lie subgroup of $N$ with Lie algebra $\g{s}=\g{n} \ominus \R \xi$. Then, we have:
\begin{enumerate}[{\rm (i)}]
\item $(R_\xi + \Ss^2_\xi) X = \frac{1}{2} ( [[X, \xi], \theta \xi] - [[X, \theta \xi]_\g{n}, \xi ])$ for each $X \in \g{s}$. \label{proposition:rxi:sxi:1}
\item $(R_\xi + \Ss^2_\xi) X = \left( \frac{1}{2} \sum_{\gamma \in \Phi} a^2_\gamma |\gamma|^2 A_{\gamma, \lambda} \right) X$ for each $X \in \g{g}_\lambda$ with $\lambda \in \Sigma^{+} \backslash \Phi$. \label{proposition:rxi:sxi:2}
\item $(R_\xi + \Ss^2_\xi) X_{\alpha} =0$ for $X_\alpha \in \g{g}_\alpha \ominus \R \xi_{\alpha}$, when $\alpha \in \Phi$, $2 \alpha$ is not a root and $A_{\alpha,\lambda} = 0$ for each $\lambda \in \Phi \backslash \{ \alpha \}$.\label{proposition:rxi:sxi:3}
\item $(R_\xi + \Ss^2_\xi) X_{\alpha} = \frac{1}{2} (a^2_\lambda |\alpha|^2 A_{\alpha, \lambda} X_\alpha - a_\lambda a_\alpha [[\theta \xi_\alpha, X_\alpha], \xi_\lambda])$ for $X_\alpha \in \g{g}_\alpha \ominus \R \xi_{\alpha}$, when $\Phi = \{ \alpha, \lambda \}$, $2 \alpha$ is not a root and $A_{\alpha, \lambda} < 0$.\label{proposition:rxi:sxi:4}
\end{enumerate}
\end{proposition}

\begin{proof}
Recall that $R_\xi$ and $\Ss_\xi$ are the Jacobi operator and the shape operator of $S$ as a submanifold of $N$, respectively. This means that we should compute them using the Levi-Civita connection $\nabla$ of $N$. Hence, we have that $\Ss_\xi X = -\nabla_X \xi$ for any $X \in \g{s}$. Now, using Lemma~\ref{lemma:levi:connection:xi}~(\ref{lemma:levi:connection:xi:3}) in the second equality below, Lemma~\ref{lemma:simplify:ric:S}~(\ref{lemma:simplify:ric:S:1}) in the second and sixth equalities below, the symmetry of the Levi-Civita connection in the third equality and Lemma~\ref{lemma:simplify:ric:S}~(\ref{lemma:simplify:ric:S:2}) in the fifth one, we get 
\begin{align*}
(R_\xi + \Ss^2_\xi) X &= \nabla_X \nabla_\xi \xi - \nabla_\xi \nabla_X \xi - \nabla_{[X, \xi]} \xi + \Ss^2_\xi X = \nabla_\xi \Ss_\xi X + \Ss_\xi [X, \xi] + \Ss^2_\xi X \\
& = [\xi, \Ss_\xi X] + \nabla_{\Ss_\xi X} \xi + \Ss_\xi [X, \xi] + \Ss^2_\xi X = - [\Ss_\xi X, \xi] -\Ss^2_\xi X + \Ss_\xi [X, \xi] + \Ss^2_\xi X \\
& = \frac{1}{2} ([[X, \xi], \xi] - [[X, \theta \xi]_\g{n}, \xi ]  - [[X, \xi], \xi] + [[X, \xi], \theta \xi]_\g{n}) \\
& =\frac{1}{2} (  - [[X, \theta \xi]_\g{n}, \xi ] + [[X, \xi], \theta \xi]),
\end{align*}
for each $X \in \g{s}$, which proves assertion~(\ref{proposition:rxi:sxi:1}). Now, let $X \in \g{g}_\lambda$ for some $\lambda \in \Sigma^{+} \backslash \Phi$. Hence, we have $[X, \theta \xi]_\g{n} = [X, \theta \xi]$. Using this in assertion~(\ref{proposition:rxi:sxi:1}), together with the Jacobi identity and Lemma~\ref{lemma:levi:connection:xi}~(\ref{lemma:levi:connection:xi:1}), we get 
\begin{align*}
(R_\xi + \Ss^2_\xi) X &= \frac{1}{2} ( [[X, \xi], \theta \xi] - [[X, \theta \xi], \xi ])\\
& = \frac{1}{2} ( -[[\theta \xi, X], \xi]- [[\xi, \theta \xi], X]- [[X, \theta \xi], \xi ])\\
&= \frac{1}{2} [[\theta \xi, \xi], X] = \sum_{\gamma \in \Phi} a^2_\gamma [H_\gamma, X] = \frac{1}{2} \sum_{\gamma \in \Phi} a^2_\gamma |\gamma|^2 A_{\gamma, \lambda}  X,
\end{align*}
which proves assertion~(\ref{proposition:rxi:sxi:2}). 

From now to the end of the proof, let us assume that $\alpha \in \Phi$, $2 \alpha$ is not a root and take $X_\alpha \in \g{g}_\alpha \ominus \R \xi_\alpha$. First, we have that $[X_\alpha, \theta \xi] \in \g{k}_0$ by virtue of Lemma~\ref{lemma:berndt:sanmartin}~(\ref{lemma:berndt:sanmartin:ii}). 

On the one hand, if $A_{\alpha,\lambda} = 0$ for all $\lambda \in \Phi \backslash \{ \alpha \}$ (equivalently, $\alpha + \lambda$ is not a root) and since $2 \alpha$ is not a root, then we deduce that $[X_\alpha, \xi] = 0$. If we use assertion~(\ref{proposition:rxi:sxi:1}) together with these two considerations, we easily get assertion~(\ref{proposition:rxi:sxi:3}).

On the other hand, let us assume that $\Phi = \{ \alpha, \lambda \}$ and $A_{\alpha, \lambda} < 0$. Let $\gamma \in \Sigma^{+}$ be a root of level two. This means that we can write $\gamma = \gamma_1 + \gamma_2$ for some $\gamma_1$, $\gamma_2 \in \Pi$. If $\beta \in \Sigma^{+}$, then $\gamma - \beta$ is in $\Sigma^{+}$ if and only if $\beta \in \{\gamma_1, \gamma_2\}$. Using this in assertion~(\ref{proposition:rxi:sxi:1}), recalling that $[X_\alpha, \theta \xi] \in \g{k}_0$ and taking Lemma~\ref{lemma:theta:alpha:new}~(\ref{lemma:theta:alpha:new:2}) and Jacobi identity into account, we get 
\begin{align*}
(R_\xi + \Ss^2_\xi) X_\alpha & = \frac{1}{2} [[X_\alpha, \xi], \theta \xi]  =  \frac{1}{2} a^2_\lambda [[X_\alpha, \xi_\lambda ], \theta \xi_\lambda] + \frac{1}{2} a_\lambda a_\alpha [[X_\alpha, \xi_\lambda ], \theta \xi_\alpha]\\
&= \frac{1}{2} a^2_\lambda |\alpha|^2 A_{\alpha, \lambda} X_\alpha - \frac{1}{2} a_\lambda a_\alpha [[\theta \xi_\alpha, X_\alpha], \xi_\lambda]. \qedhere
\end{align*}
\end{proof}

In the above result, we have investigated the endomorphism $R_\xi + \Ss^2_\xi$ of $\g{s}$ when restricted to the first two addends of decomposition~\eqref{equation:tangent:decomposition} of $\g{s}$. Now, we study its restriction to the remaining term of~\eqref{equation:tangent:decomposition}. Recall the definition of the unit vector $\eta_{\alpha, \lambda}$ of $\g{s}$ in~\eqref{equation:definitio:eta}, for each pair of roots $\alpha$, $\lambda \in \Phi$. We have the following

\begin{proposition}\label{propostion:rs:eta}
For each $\nu \in \Phi$, let $\Phi_\nu$ be the set of roots $\mu \in \Phi \backslash \{\nu\}$ such that $\nu + \mu$ is a root. Let $\alpha$, $\lambda$ be different roots in $\Phi$. Then, we have: 
\begin{enumerate}[{\rm (i)}]
\item $(R_\xi + \Ss^2_\xi) \eta_{\alpha, \lambda} = \frac{1}{2} l_{\alpha, \lambda} \left(  \sum_{\nu \in \Phi_\alpha} a_\lambda a_\nu |\alpha|^2 A_{\alpha, \nu} l_{\alpha, \nu}^{-1} \eta_{\alpha,\nu} + \sum_{\mu \in \Phi_\lambda} a_\alpha a_\mu |\lambda|^2 A_{\lambda, \mu} l_{\mu, \lambda}^{-1} \eta_{\mu,\lambda} \right)$. \label{propostion:rs:eta:1}
\item $(R_\xi + \Ss^2_\xi) \eta_{\alpha, \lambda} = \frac{1}{2} (a^2_\alpha + a^2_\lambda) |\alpha|^2 A_{\alpha, \lambda} \eta_{\alpha, \lambda}$ if $A_{\alpha, \nu} = A_{\lambda, \nu} = 0$ for all $\nu \in \Phi \backslash \{ \alpha, \lambda \}$.\label{propostion:rs:eta:2}
\end{enumerate}
\end{proposition}

\begin{proof}
Let us fix two distinct roots $\alpha$, $\lambda \in \Phi$. First, using that $\gamma_1 - \gamma_2$ is not a root for any $\gamma_1$, $\gamma_2 \in \Pi$ and Lemma~\ref{lemma:berndt:sanmartin}~(\ref{lemma:berndt:sanmartin:i}), we deduce  
\begin{align}\label{rs:eta:1}
[\eta_{\alpha, \lambda}, \theta \xi]_\g{n} & = l_{\alpha, \lambda} (a_\lambda[\xi_\alpha, \theta \xi] - a_\alpha [\xi_\lambda, \theta \xi])_\g{n} = l_{\alpha, \lambda} a_\lambda a_\alpha(  [\xi_\alpha, \theta \xi_\alpha] - [\xi_\lambda, \theta \xi_\lambda])_\g{n} \nonumber  \\
& = - 2 l_{\alpha, \lambda} a_\lambda a_\alpha ( H_\alpha -  H_\lambda)_\g{n} = 0.
\end{align}
If $\beta, \gamma \in \Sigma^{+}$ with $\gamma= \gamma_1 + \gamma_2$ of level two (i.e. $\gamma_1$, $\gamma_2 \in \Pi$), then $\gamma - \beta \in \Sigma^{+}$ if and only if $\beta \in \{\gamma_1, \gamma_2\}$. Using this in the third equality below, Proposition~\ref{proposition:rxi:sxi}~(\ref{proposition:rxi:sxi:1}),~\eqref{rs:eta:1} and Lemma~\ref{lemma:theta:alpha:new}~(\ref{lemma:theta:alpha:new:1})-(\ref{lemma:theta:alpha:new:2}), we get 
\begin{align*}
(R_\xi + \Ss^2_\xi) \eta_{\alpha, \lambda} & =  \frac{1}{2} [[\eta_{\alpha,\lambda}, \xi], \theta \xi]\\
& = \frac{l_{\alpha, \lambda}}{2} \left(\sum_{\nu \in \Phi_\alpha} a_\lambda a_\nu [[\xi_\alpha, \xi_\nu], \theta \xi] - \sum_{\mu \in \Phi_\lambda} a_\alpha a_\mu [[\xi_\lambda, \xi_\mu], \theta \xi]   \right)\\
& = \frac{l_{\alpha, \lambda}}{2} \left(\sum_{\nu \in \Phi_\alpha} a_\lambda a_\nu [[\xi_\alpha, \xi_\nu], a_\alpha \theta \xi_\alpha + a_\nu \theta \xi_\nu] - \sum_{\mu \in \Phi_\lambda} a_\alpha a_\mu [[\xi_\lambda, \xi_\mu], a_\lambda \theta \xi_\lambda + a_\mu \theta \xi_\mu]  \right)\\
& = \frac{l_{\alpha, \lambda}}{2} \left(\sum_{\nu \in \Phi_\alpha} a_\lambda a_\nu |\alpha|^2 A_{\alpha, \nu} ( - a_\alpha \xi_\nu + a_\nu \xi_\alpha ) + \sum_{\mu \in \Phi_\lambda} a_\alpha a_\mu |\lambda|^2 A_{\lambda, \mu} (  a_\lambda \xi_\mu - a_\mu \xi_\lambda ) \right) \\
& = \frac{l_{\alpha, \lambda}}{2} \left(  \sum_{\nu \in \Phi_\alpha} a_\lambda a_\nu |\alpha|^2 A_{\alpha, \nu} l_{\alpha, \nu}^{-1} \eta_{\alpha,\nu} + \sum_{\mu \in \Phi_\lambda} a_\alpha a_\mu |\lambda|^2 A_{\lambda, \mu} l_{\mu, \lambda}^{-1} \eta_{\mu,\lambda} \right).
\end{align*}
This proves~(\ref{propostion:rs:eta:1}). If $A_{\alpha, \nu} = A_{\lambda, \nu} = 0$ for all $\nu \in \Phi \backslash \{ \alpha, \lambda \}$, then $\Phi_{\alpha} \subset \{ \lambda \}$ and $\Phi_{\lambda} \subset \{ \alpha \}$. Using this and $|\alpha|^2 A_{\alpha, \lambda} = | \lambda|^2 A_{\lambda, \alpha}$ in assertion~(\ref{propostion:rs:eta:1}), we obtain assertion~(\ref{propostion:rs:eta:2}).
\end{proof}

After the analysis of term $R_\xi + \Ss^2_\xi$ of the endomorphism $D$ defined in~\eqref{definition:D}, we will focus now on the remaining non-trivial term of $D$, that is, $\ad(\mathcal{H})_{\rvert{\g{s}}}^\top$. Recall that $\mathcal{H} = \sum_{\lambda \in \Sigma^{+} } \dim \g{g}_\lambda H_\lambda$, as obtained in Lemma~\ref{lemma:ricN}~({\ref{lemma:ricN:4}}). 

Obtaining a general, manageable and explicit expression of $(\ad(\mathcal{H})_{\rvert_{\g{s}}})^\top$ when applied to an element of the form $\eta_{\alpha, \lambda}$, with $\alpha$, $\lambda \in \Phi$ of different length, is complicated. Actually, we will need to impose some extra conditions on $\Phi$ in order to calculate $(\ad(\mathcal{H})_{\rvert_{\g{s}}})^\top$ efficiently.

We introduce first an auxiliary result related to the root system $\Sigma$ of $M = G/K$. Fix a root $\alpha \in \Sigma^{+}$. We will define an equivalence relation $\sim_{\alpha}$ in $\Sigma^{+}$ depending on $\alpha$. Two positive roots $\gamma_1$, $\gamma_2$ are $\alpha$-related if $\gamma_1-\gamma_2 = n \alpha$ for some integer number $n$. Equivalently, two positive roots $\gamma_1$, $\gamma_2 \in \Sigma^{+}$ are $\alpha$-related if they belong to the same $\alpha$-string~\cite[p.~152]{K}. For each $\lambda \in \Sigma^{+}$, we will write $[\lambda]_\alpha$ for the equivalence class of $\lambda$, that is, for set of positive roots in the $\alpha$-string of $\lambda$.

\begin{lemma}\label{lemma:sum:strings}
Let $\alpha \in \Pi$. Then $\sum_{\gamma \in \Sigma^{+}} \dim \g{g}_\gamma A_{\alpha, \gamma} = 2 \dim \g{g}_\alpha + 4 \dim \g{g}_{2\alpha}$.
\end{lemma}

\begin{proof}
Let $\lambda$ be a positive root non-proportional to $\alpha$ and of minimum level in its $\alpha$-string. Then any root in the $\alpha$-string of $\lambda$ is positive. This follows from the fact that $\alpha$ and $\lambda$ are non-proportional and that the coefficients derived from the expression of a negative root as a linear combination of simple roots are all non-positive by virtue of~\cite[Proposition~2.49]{K}.

Now, according to~\cite[Proposition~2.48]{K}, we have that $[\lambda]_\alpha =  \{ \lambda + n \alpha : n= 0, \dots, k \}$ for some $k \in \{0,1,2,3\}$; note that we have $-k = A_{\alpha, \lambda}$ and then $A_{\alpha, \lambda + k \alpha} =    A_{\alpha, \lambda} + 2k = -  A_{\alpha, \lambda}$. Moreover, we have $A_{\alpha, \lambda + \alpha} = 0$ if $k=2$, and $A_{\alpha, \lambda + \alpha} = - A_{\alpha, \lambda + 2\alpha}$ if $k=3$. This, recalling that any root in the $\alpha$-string of $\lambda$ is positive, proves that 
\begin{equation}\label{equation:sum:strings}
\sum_{\gamma \in [\lambda]_\alpha} A_{\alpha, \gamma} =0
\end{equation}
for each positive root $\lambda$ non-proportional to $\alpha$. If $k=3$, then $\alpha$ and $\lambda$ generate a $G_2$ simple system. Then, according to~\cite[p.~339]{Jurgen}, we have $\dim \g{g}_{\gamma_1} = \dim \g{g}_{\gamma_2}$ for all $\gamma_1$, $\gamma_2 \in [\lambda]_\alpha$. If $k = 1$, then we get $\dim \g{g}_\lambda = \dim \g{g}_{\lambda + \alpha}$ by virtue of~\cite[Lemma~2.2~(i)]{BS18}. If $k=2$, we get $\dim \g{g}_{\lambda} = \dim \g{g}_{\lambda + 2 \alpha}$ by virtue of~\cite[Lemma~2.2~(ii)]{BS18}; in this case $\dim \g{g}_\lambda$ might be different from $\dim \g{g}_{\lambda+\alpha}$, but recall that $A_{\alpha, \lambda + \alpha} = 0$. Hence, using these considerations on the dimensions of the root spaces along with~\eqref{equation:sum:strings}, we get 
\begin{equation}\label{equation:sum:strings:2}
\sum_{\gamma \in [\lambda]_\alpha} \dim \g{g}_\gamma A_{\alpha, \gamma} =0
\end{equation}
for each positive root $\lambda$ non-proportional to $\alpha$. Note that if $\nu \in \Sigma^{+}$ is proportional to $\alpha$, then $[\nu]_\alpha$ is either $\{ \alpha \}$ or $\{ \alpha, 2\alpha \}$. Finally, using this, $\Sigma^{+} = \bigsqcup_{ [\lambda]_\alpha \in \Sigma^{+} / \sim_{\alpha}  } [\lambda]_\alpha$ and~\eqref{equation:sum:strings:2}, we get 
\[
\sum_{\gamma \in \Sigma^{+}} \dim \g{g}_\gamma A_{\alpha, \gamma} = \sum_{[\lambda]_\alpha \in \Sigma^{+}/ \sim_\alpha}  \sum_{\gamma \in [\lambda]_\alpha}  \dim \g{g}_\gamma A_{\alpha, \gamma}  = 2 \dim \g{g}_\alpha + 4 \dim \g{g}_{2\alpha},
\]
which concludes the proof.
\end{proof}

\begin{remark}
Taking $\alpha$ as the highest level (non-simple) root of an $A_2$ system allows to see that Lemma~\ref{lemma:sum:strings} is not true if $\alpha \in \Sigma^{+} \backslash \Pi$. This is because, in general, the $\alpha$-string of a positive root may contain negative roots and then~\eqref{equation:sum:strings} does not need to be true.
\end{remark}
We are now in position to calculate the term $\ad(\mathcal{H})_{\rvert{\g{s}}}^\top$.
\begin{proposition}\label{proposition:ad}
Let $\nu = \sum_{\beta \in \Pi} b_\beta \beta \in \Sigma^{+}$ be a positive root. Let $\alpha$, $\lambda \in \Phi$ be distinct roots. Then, we have:
\begin{enumerate}[{\rm (i)}]
\item $(\ad(\mathcal{H}) X)^\top = \ad(\mathcal{H}) X = \sum_{\beta \in \Pi} b_\beta |\beta|^2 (\dim \g{g}_\beta + 2 \dim \g{g}_{2\beta}) X$ for each $X$ in $\g{g}_{\nu} \cap \g{s}$.  \label{proposition:ad:1}
\item $(\ad(\mathcal{H}) \eta_{\alpha, \lambda})^\top = |\alpha|^2 \dim \g{g}_{\alpha} \eta_{\alpha, \lambda}$ if $\alpha$ and $\lambda$ have the same length.\label{proposition:ad:2}
\item $(\ad(\mathcal{H}) \eta_{\alpha, \lambda})^\top =(a^2_\lambda |\alpha|^2 (\dim \g{g}_\alpha + 2\dim \g{g}_{2\alpha}) + a^2_\alpha |\lambda|^2 \dim \g{g}_\lambda)  \eta_{\alpha, \lambda}$ if $\Phi = \{ \alpha, \lambda \}$ and $|\alpha|\leq|\lambda|$.  \label{proposition:ad:3}
\end{enumerate}
\end{proposition}
\begin{proof}
Recall from Lemma~\ref{lemma:ricN}~(\ref{lemma:ricN:4}) that $\mathcal{H} = \sum_{\gamma \in \Sigma^{+}} \dim \g{g}_\gamma H_\gamma$. Let $X$ be a vector in $\g{g}_\nu$. Now, using Lemma~\ref{lemma:sum:strings} for each $\beta \in \Pi$, we obtain 
\begin{align}\label{equation:ad}
\ad(\mathcal{H}) X & = \sum_{\gamma \in \Sigma^{+}} \dim \g{g}_\gamma [H_\gamma, X] =  \sum_{\gamma \in \Sigma^{+}} \dim \g{g}_\gamma \langle \gamma, \nu \rangle X = \sum_{\gamma \in \Sigma^{+}} \sum_{\beta \in \Pi} b_\beta \dim \g{g}_\gamma \langle \gamma, \beta \rangle X \nonumber \\
& = \sum_{\gamma \in \Sigma^{+}} \sum_{\beta \in \Pi} b_\beta \dim \g{g}_\gamma \frac{|\beta|^2}{2} A_{\beta, \gamma} X =  \sum_{\beta \in \Pi} b_\beta \frac{|\beta|^2}{2} \sum_{\gamma \in \Sigma^{+}} \dim \g{g}_\gamma A_{\beta, \gamma} X \nonumber \\
& = \sum_{\beta \in \Pi} b_\beta |\beta|^2 (\dim \g{g}_\beta + 2\dim \g{g}_{2\beta} ) X.
\end{align}
In particular, if we assume that $X \in \g{g}_\nu \cap \g{s}$, then~\eqref{equation:ad} implies that $\ad(\mathcal{H}) X  = (\ad(\mathcal{H}) X)^\top$ and assertion~(\ref{proposition:ad:1}) follows. Let us assume, without loss of generality, that $|\alpha| \leq |\lambda|$, which implies that $2 \lambda$ cannot be a root, since $\alpha$ and $\lambda$ are both simple. Using~\eqref{equation:ad}, we deduce 
\begin{align}\label{equation:ad:eta}
\ad(\mathcal{H}) \eta_{\alpha, \lambda} & =  l_{\alpha, \lambda} ( a_\lambda \ad(\mathcal{H}) \xi_\alpha - a_\alpha \ad(\mathcal{H}) \xi_\lambda) \nonumber \\
& = l_{\alpha, \lambda} (a_\lambda |\alpha|^2 (\dim \g{g}_\alpha + 2\dim \g{g}_{2\alpha}) \xi_\alpha - a_\alpha |\lambda|^2 \dim \g{g}_\lambda \xi_\lambda). 
\end{align}
First, let us assume that $\alpha$ and $\lambda$ have the same length. Thus, $2 \alpha$ is not a root and $\dim \g{g}_\alpha = \dim \g{g}_\lambda$. Hence, using~\eqref{equation:ad:eta} we deduce that $(\ad(\mathcal{H})\eta_{\alpha, \lambda})^\top = |\alpha|^2 \dim \g{g}_\alpha\eta_{\alpha, \lambda}^\top =|\alpha|^2 \dim \g{g}_\alpha\eta_{\alpha, \lambda}$, which proves assertion~(\ref{proposition:ad:2}). 

Finally, put $\Phi = \{ \alpha, \lambda \}$ with $|\alpha| \leq |\lambda|$. This means that $\eta_{\alpha, \lambda}$ is the unique element in the last addend of decomposition~\eqref{equation:tangent:decomposition} and that $l_{\alpha, \lambda} = 1$. Thus, according to~\eqref{equation:ad:eta}, we have that $\ad(\mathcal{H}) \eta_{\alpha, \lambda}$ is orthogonal to all the elements in decomposition~\eqref{equation:tangent:decomposition} except maybe $\eta_{\alpha, \lambda}$. Hence, $(\ad(\mathcal{H}) \eta_{\alpha, \lambda})^\top = \langle \ad(\mathcal{H}) \eta_{\alpha, \lambda}, \eta_{\alpha, \lambda}  \rangle \eta_{\alpha, \lambda}$. Using this and~\eqref{equation:ad:eta}, claim~(\ref{proposition:ad:3}) follows.
\end{proof}
\section{Systematic investigation}\label{section:classification}
Let $S$ be a codimension one Lie subgroup of a nilpotent Iwasawa group $N$. Our aim is to determine in which cases $S$ is a Ricci soliton when considered with the induced metric. Recall that $S$ is a Ricci soliton if and only if for some real value $c$ the endomorphism $D$ of $\g{s}$ defined in~\eqref{definition:D} is a derivation, that is, if for some real value $c$ and any $X$, $Y \in \g{s}$ we have
\begin{equation}\label{equation:derivation}
	D[X,Y] = [DX, Y] + [X, DY].
\end{equation}

In this section, we start with the classification problem itself. More precisely, recall that $S$ determines a subset $\Phi$ of the set $\Pi$ of simple roots, as follows from Lemma~\ref{lemma:subalgebra:s:xi}. This section is mainly devoted to see that if $S$ is a Ricci soliton and $\Phi$ contains more than one simple root, then $M \cong AN$ must be one of the symmetric spaces of non-compact type that appear from item~(\ref{main:theorem:3}) to item~(\ref{main:theorem:6}) of the Main Theorem.  In order to so, here is how we will proceed: we will assume that $S$ is a Ricci soliton and that $\Phi$ contains at least two simple roots and then, for certain suitable choices of $X$, $Y \in \g{s}$, we will make extensive use of~\eqref{equation:derivation} taking into account the description of $D$ that follows from Section~\ref{section:ricci:operator}. This will lead us to several strong constraints on the set $\Phi$ and, as a consequence, on the root system $\Sigma$ of the irreducible symmetric space of non-compact type $M$. This will allow us to deduce that if $S$ is a Ricci soliton, then $M$ must be one of the symmetric spaces contained in the Main Theorem.

The first constraint is related to the rank of the symmetric space and concerns the real number $c$ in the definition of $D$ in~\eqref{definition:D}. Indeed, if $S$ is a Ricci soliton and the rank of $M$ is big enough, then we will have $c=0$. We make this idea more precise in the following two results.

\begin{lemma}\label{lemma:c}
We have that equation~\eqref{equation:derivation} holds for any $X$, $Y \in \bigoplus_{\gamma \in \Sigma^{+} \backslash \Phi  } \g{g}_\gamma$ if and only if one of the following conditions holds:
\begin{enumerate}[{\rm (i)}]
\item $\nu + \mu$ is not a root for any $\nu$, $\mu \in \Sigma^{+} \backslash \Phi$. \label{lemma:c:1} 
\item $\nu + \mu$ is a root for some $\nu$, $\mu \in \Sigma^{+} \backslash \Phi$ and $c = 0$. \label{lemma:c:2}
\end{enumerate}
\end{lemma}
\begin{proof}
Let $\alpha = \sum_{\beta \in \Pi} b_\beta \beta$, $\lambda = \sum_{\beta \in \Pi} c_\beta \beta$ be two arbitrary roots in $\Sigma^{+} \backslash \Phi$. Take $X$ and $Y$ in the root spaces $\g{g}_\alpha$ and $\g{g}_\lambda$, respectively. Using definition~\eqref{definition:D}, Proposition~\ref{proposition:ad}~(\ref{proposition:ad:1}) and Proposition~\ref{proposition:rxi:sxi}~(\ref{proposition:rxi:sxi:2}), we get 
\begin{align*}
D[X, Y] & = \left( \sum_{\beta \in \Pi}  (b_\beta + c_\beta) |\beta|^2 (\dim \g{g}_\beta + 2\dim \g{g}_{2\beta}) - \frac{1}{2} \sum_{\gamma \in \Phi} a^2_\gamma |\gamma|^2 A_{\gamma, \alpha+ \lambda} +c \right) [X, Y], \\
[DX, Y] & = \left( \sum_{\beta \in \Pi}  b_\beta |\beta|^2 (\dim \g{g}_\beta + 2\dim \g{g}_{2\beta}) - \frac{1}{2} \sum_{\gamma \in \Phi} a^2_\gamma |\gamma|^2 A_{\gamma, \alpha} +c \right) [X, Y], \\
[X, DY] & = \left( \sum_{\beta \in \Pi}  c_\beta |\beta|^2 (\dim \g{g}_\beta + 2\dim \g{g}_{2\beta}) - \frac{1}{2} \sum_{\gamma \in \Phi} a^2_\gamma |\gamma|^2 A_{\gamma, \lambda} +c \right) [X, Y].
\end{align*}
From these three equations we have that~\eqref{equation:derivation} holds for any $X \in \g{g}_\alpha$ and any $Y \in \g{g}_\lambda$ if and only if: $[X, Y] = 0$ for any $X \in \g{g}_\alpha$ and any $Y \in \g{g}_\lambda$, i.e. $\lambda + \alpha$ is not a positive root, as follows from Lemma~\ref{lemma:bracket:root:spaces}~(\ref{lemma:bracket:root:spaces:1}); or $[X, Y] \neq 0$ for some $X \in \g{g}_\alpha$ and some $Y \in \g{g}_\lambda$ (thus $\lambda + \alpha$ is a root) and $c = 0$.
\end{proof}

\begin{corollary}\label{corollary:c}
If $S$ is a Ricci soliton and one of the following conditions holds, then~$c = 0$:
\begin{enumerate}[{\rm (i)}]
\item $\rank M \geq 3$, except if $M$ is of type $A_3$ and $\Phi$ contains at least two roots. \label{corollary:c:1}
\item $M$ is of type $G_2$ or of type $BC_2$. \label{corollary:c:2}
\end{enumerate}
\end{corollary}
\begin{proof}
Since $S$ is a Ricci soliton by assumption, then equation~\eqref{equation:derivation} must hold, in particular, for any $X$, $Y \in \bigoplus_{\gamma \in \Sigma^{+} \backslash \Phi  } \g{g}_\gamma$. Therefore, we proceed as follows for this proof: for each one of the cases of this corollary we will find roots $\alpha$, $\lambda \in \Sigma^{+} \backslash \Phi$ such that $\alpha + \lambda$ is a root, and then using Lemma~\ref{lemma:c}~(\ref{lemma:c:2}) we will obtain $c = 0$, as desired. In general, in order to make sure that the selected $\alpha$ and $\lambda$ do not belong to $\Phi \subset \Pi$, we will take them to be non-simple.

First, let us assume that $\Sigma$ contains an $A_4$ root subsystem generated by the simple subsystem $\{ \alpha_1, \alpha_2, \alpha_3, \alpha_4\}$. Then, take $\alpha = \alpha_1 + \alpha_2$ and $\lambda = \alpha_3 + \alpha_4$. This proves~(\ref{corollary:c:1}) for symmetric spaces of type $A_n$, $D_{n+1}$, $n \geq 4$, $E_6$, $E_7$ and $E_8$. 

On the one hand, let us assume that $\Sigma$ contains a $D_4$ root subsystem generated by the simple subsystem $\{ \alpha_1, \alpha_2, \alpha_3, \alpha_4\}$, where $\alpha_1$, $\alpha_3$ and $\alpha_4$ are orthogonal to each other and all of them connected to $\alpha_2$ in the Dynkin diagram. Then, take $\alpha = \alpha_1 + \alpha_2$ and $\lambda = \alpha_2 + \alpha_3 + \alpha_4$. This proves assertion~(\ref{corollary:c:1}) for symmetric spaces of type $D_n$, $n \geq 4$.

On the other hand, let us assume that $\Sigma$ contains a $B_3$, $C_3$ or $BC_3$ root subsystem generated by the simple subsystem $\{ \alpha_1, \alpha_2, \alpha_3\}$, where $|\alpha_1| = |\alpha_2| \neq |\alpha_3|$. Then, take the roots $\alpha = \alpha_2 + \alpha_3$ and $\lambda = \alpha_1 + \alpha_2 + \alpha_3$ if $|\alpha_2|> |\alpha_3|$, or the roots $\alpha = \alpha_1 + \alpha_2$ and $\lambda = \alpha_2 + \alpha_3$ otherwise. This proves~(\ref{corollary:c:1}) for symmetric spaces of type $B_n$, $C_n$, $BC_n$, $n \geq 3$, and $F_4$. 

In order to finish the proof of assertion~(\ref{corollary:c:1}), let $\Pi = \{ \alpha_0, \mu, \alpha_1 \}$ be the set of simple roots of an $A_3$ root system with $A_{\alpha_0, \alpha_1} = 0$, and let $\Phi = \{ \nu \}$ for some $\nu \in \Pi$. Then, we have that $\alpha_k$ is in $\Sigma^{+} \backslash \Phi$ for some fixed $k \in \{0,1 \}$. Take the roots $\alpha = \alpha_k$ and $\lambda=\alpha_{k+1} + \mu$, with indices modulo 2. This finishes the proof of assertion~(\ref{corollary:c:1}).

Let $\Pi = \{ \alpha_1, \alpha_2 \}$ be the set of simple roots of a $G_2$ root system with $|\alpha_1| > |\alpha_2|$. Then, take the roots $\alpha = \alpha_1 + \alpha_2$ and $\lambda = \alpha_1 + 2 \alpha_2$.

Finally, let $\Pi = \{ \alpha_1, \alpha_2 \}$ be the set of simple roots of a $BC_2$ root system with $|\alpha_1| > |\alpha_2|$. Then, take the roots $\alpha = \lambda = \alpha_1 + \alpha_2$. This finishes the proof of assertion~(\ref{corollary:c:2}).
\end{proof}

Let $\Psi$ be a subset of the set of simple roots $\Pi$. From now on, we will say that $\Psi$ contains a \emph{connected subset} of $r$ elements if there exists a subset $\{ \alpha_1, \dots, \alpha_r \}$ of $\Psi$ such that $A_{\alpha_j, \alpha_{j+1}} < 0 $ for each $j \in \{1, \dots, r-1 \}$. This means that $\Psi$ contains a subset of simple roots that are in correspondence with a connected subgraph of the Dynkin diagram of~$\Pi$. 

In the following lines, we will focus on understanding the structure of the subset $\Phi$. We start by giving an upper bound for the number of elements of a connected subset of $\Phi$, provided that $S$ is a Ricci soliton. For each $\mu$ in $\Phi$, we denote by $\Phi_\mu$ the set of roots $\nu \in \Phi \backslash \{ \mu \}$ such that $\nu + \mu$ is a root, as we did in Proposition~\ref{propostion:rs:eta}.

\begin{proposition}\label{proposition:3:roots}
Assume that $S$ is a Ricci soliton. Then, any connected subset of $\Phi$ contains at most two elements.
\end{proposition}

\begin{proof}
In order to prove this result, we will assume that $\Phi$ contains (at least) one connected subset of three or more elements and then, making use of~\eqref{equation:derivation} for some suitable $X$, $Y \in \g{s}$, we will get a contradiction with the fact that $S$ is a Ricci soliton.

If $\mu$ is in $\Phi$, recall that $\Phi_\mu$ denotes the set of roots $\nu \in \Phi \backslash \{ \mu \}$ such that $\nu + \mu$ is a root. Without loss of generality, we can assume that $\{ \lambda, \alpha, \beta \}$ is a connected subset of three elements of $\Phi$, with $A_{\lambda, \alpha}$, $A_{\alpha, \beta} < 0$, $|\alpha| = |\lambda|$ and $\Phi_{\lambda} = \{ \alpha \}$. Let us justify this claim.

Let $\{ \lambda, \alpha, \beta \}$ be a connected subset of three elements of $\Phi$, with $A_{\lambda, \alpha}$, $A_{\alpha, \beta} < 0$. Looking at the possible configurations of a Dynkin diagram (see~\cite[p.~336-340]{Jurgen}), we deduce that $\lambda$ and $\alpha$ have the same length, or $\alpha$ and $\beta$ do (or both things are true). Let us assume that $\lambda$ and $\alpha$ have the same length without loss of generality. If $\Phi_{\lambda} = \{ \alpha \}$ we are done. Otherwise, we will have that $\lambda$ is also connected to a simple root $\nu \in \Phi$ different from $\alpha$ in the Dynkin diagram of $\Pi$. On the one hand, if $\lambda$ and $\nu$ have the same length, then we examine whether $\Phi_{\nu} = \{ \lambda \}$. If it holds, we take the connected subset $\{ \nu, \lambda, \alpha \}$ of $\Phi$ and our claim follows. Otherwise, we continue with this recursive argument. On the other hand, if $\lambda$ and $\nu$ have different length, then the possible configurations of a Dynkin diagram imply that $\beta$ and $\alpha$ have the same length. Then, we examine whether $\Phi_{\beta} = \{ \alpha \}$. If it holds we are done. If not, we continue with this kind of recursive argument.

In conclusion, we assume without loss of generality that $\{ \lambda, \alpha, \beta \}$ is a connected subset of $\Phi$ with $ A_{\lambda, \alpha}$, $A_{\alpha, \beta} < 0$, $|\alpha| = |\lambda|$ and $\Phi_{\lambda} = \{ \alpha \}$. Using~\cite[Proposition~2.48~(d)-(g)]{K}, we get $A_{\lambda, \alpha} = A_{\alpha, \lambda} = -1$, and also that $2\alpha$, $2\lambda$, $2\lambda + \alpha$ and $2\alpha + \lambda$ are not roots. Now, we will examine~\eqref{equation:derivation} for $X = \eta_{\alpha, \lambda}$ and a unit vector $Y \in \g{g}_{\alpha + \lambda}$. We have that $[\eta_{\alpha, \lambda}, Y] = 0$ since $2\lambda + \alpha$ and $2\alpha + \lambda$ are not roots. Moreover, $DZ$ is proportional to $Z$ for any $Z \in \g{g}_{\lambda+\alpha}$, by means of~\eqref{definition:D}, Proposition~\ref{proposition:ad}~(\ref{proposition:ad:1}) and Proposition~\ref{proposition:rxi:sxi}~(\ref{proposition:rxi:sxi:2}). Thus, we deduce that $D[\eta_{\alpha, \lambda}, Y] = [\eta_{\alpha, \lambda}, DY] = 0$. If we see that $[D\eta_{\alpha, \lambda}, Y] \neq 0$ for any real value $c$, we will get the desired contradiction.

Recall that $\Phi_{\lambda} = \{ \alpha \}$. Using this, $[\eta_{\alpha, \lambda}, Y] = 0$, definition~\eqref{definition:D}, Proposition~\ref{proposition:ad}~(\ref{proposition:ad:2}) and Proposition~\ref{propostion:rs:eta}~(\ref{propostion:rs:eta:1}) we deduce 
\begin{equation*}
[D \eta_{\alpha, \lambda}, Y] = - \frac{1}{2} l_{\alpha, \lambda}   \sum_{\nu \in \Phi_\alpha \backslash \{ \lambda \} } a_\lambda a_\nu |\alpha|^2 A_{\alpha, \nu} l_{\alpha, \nu}^{-1} [\eta_{\alpha,\nu}, Y].
\end{equation*}
Note that $\beta \in  \Phi_\alpha \backslash \{ \lambda \}$. Moreover, we have that $[\eta_{\alpha,\beta}, Y] = -l_{\alpha, \beta} a_\alpha [\xi_\beta, Y]$ is a non-zero vector in $\g{g}_{\lambda +\alpha + \beta}$ by means of Lemma~\ref{lemma:theta:alpha:new}~(\ref{lemma:theta:alpha:new:3}). The set $\Phi_\alpha \backslash \{ \lambda \}$ might contain just another root $\mu$ different from $\beta$. However, $[\eta_{\alpha, \mu}, Y]$ would belong to $\g{g}_{\alpha+\mu+\lambda}$, and $[D\eta_{\alpha, \lambda}, Y] \neq 0$.
\end{proof}

The above result gives two as an upper bound for the number of elements of a connected subset of $\Phi$, provided that $S$ is a Ricci soliton. The next result shows that there is at most one connected subset of two elements in $\Phi$, and gives some extra conditions on such subset.

\begin{proposition}\label{proposition:2:roots}
Assume that $S$ is a Ricci soliton. Then, one of the following conditions holds:
\begin{enumerate}[{\rm (i)}]
\item $\Phi = \Pi$ is the set of simple roots of a root system of type $A_2$.\label{proposition:2:roots:1}
\item $\Phi = \{ \alpha, \lambda \} \cup \Phi^{\prime}$ for some $\Phi^{\prime} \subset \Pi$, with $\alpha$, $\lambda \in \Pi$ of different length and connected in the Dynkin diagram, and any root $\nu$ in $\Phi^\prime$ orthogonal to any root in $\Phi \backslash \{\nu\}$.  \label{proposition:2:roots:2}
\item The roots in $\Phi$ are orthogonal to each other. \label{proposition:2:roots:3}
\end{enumerate}
\end{proposition}

\begin{proof}
A connected subset of $\Phi$ contains at most two elements by means of Proposition~\ref{proposition:3:roots}. If $\Phi$ contains (at least) two connected subsets of two elements, looking at the possible configurations of a Dynkin diagram (see~\cite[p.~336-340]{Jurgen}), we deduce that in one of them both roots have the same length. Hence, for this proof we will assume that $\Phi$ contains (at least) a connected subset of two roots of the same length, otherwise  either~(\ref{proposition:2:roots:2}) or~(\ref{proposition:2:roots:3}) must hold, and we will see that this leads us to~(\ref{proposition:2:roots:1}).

Hence, let $\alpha$, $\lambda \in \Phi$ be roots generating a root subsystem of $\Sigma$ of type $A_2$. Thus, $2\alpha$ and $2\lambda$ are not roots. We will assume that $\rank M \geq 3$ and we will get a contradiction with the fact that $S$ is a Ricci soliton, and this will lead us to~(\ref{proposition:2:roots:1}). Since $\Pi$ is irreducible and we assume $\rank M \geq 3$, let $\beta$ be, without loss of generality, a root in $\Pi \backslash \{ \lambda \}$ connected to $\alpha$ in the Dynkin diagram. We must have that  $\beta \in \Pi \backslash \Phi$, since otherwise $\{ \lambda, \alpha, \beta \}$ would be a connected subset of three elements of $\Phi$, and then $S$ could not be a Ricci soliton by virtue of Proposition~\ref{proposition:3:roots}.

Since $S$ is a Ricci soliton by assumption and the sum of the roots $\beta$, $\lambda + \alpha \in \Sigma^{+} \backslash \Phi$ is also a root, then $c = 0$ in~\eqref{definition:D} as follows from Lemma~\ref{lemma:c}~(\ref{lemma:c:2}). Now, we will check equality~\eqref{equation:derivation} for $X = \eta_{\alpha, \lambda}$ and a non-zero vector $Y$ in $\g{g}_\beta$. Note that $2\beta$ might be a root, but recall that neither $2\alpha$ nor $2\lambda$ can. Using~\eqref{definition:D} recalling that $c = 0$, Proposition~\ref{proposition:ad}~(\ref{proposition:ad:1}) and Proposition~\ref{proposition:rxi:sxi}~(\ref{proposition:rxi:sxi:2}), we get 
\begin{align*}
D[\eta_{\alpha, \lambda}, Y] &= \left(|\alpha|^2 \dim \g{g}_\alpha + |\beta|^2 (\dim \g{g}_\beta + 2 \dim \g{g}_{2\beta}) - \frac{1}{2}\sum_{\gamma \in \Phi} a^2_\gamma |\gamma|^2 A_{\gamma, \alpha + \beta}\right)[\eta_{\alpha, \lambda}, Y],\\
[\eta_{\alpha, \lambda}, DY] &= \left(|\beta|^2 (\dim \g{g}_\beta + 2 \dim \g{g}_{2\beta}) - \frac{1}{2}\sum_{\gamma \in \Phi} a^2_\gamma |\gamma|^2 A_{\gamma, \beta}  \right)[\eta_{\alpha, \lambda}, Y].
\end{align*}
Similarly, using Proposition~\ref{proposition:ad}~(\ref{proposition:ad:2}) and Proposition~\ref{propostion:rs:eta}~(\ref{propostion:rs:eta:2}) we deduce 
\begin{equation*}
[D\eta_{\alpha, \lambda}, Y] = \left(|\alpha|^2 \dim \g{g}_\alpha - \frac{1}{2} (a^2_\alpha+ a^2_\lambda) |\alpha|^2 A_{\alpha, \lambda} \right)[\eta_{\alpha, \lambda}, Y].
\end{equation*}
Since $[\eta_{\alpha, \lambda}, Y] =  l_{\alpha, \lambda} a_\lambda [\xi_\alpha, Y]$ is a non-zero vector in $\g{g}_{\alpha+\beta}$ by virtue of Lemma~\ref{lemma:theta:alpha:new}~(\ref{lemma:theta:alpha:new:3}), putting the three expressions above into equation~\eqref{equation:derivation}, we get $a_\alpha |\alpha| = 0$, which is a contradiction. This contradiction comes from the assumption $\rank M \geq 3$. Hence, $\rank M = 2$ and $\Pi = \Phi$ generates a root system of type $A_2$. This concludes the proof.
\end{proof}

We go further in the examination of the case in which $\Phi$ contains orthogonal roots. We start with an auxiliary result. For each $\alpha \in \Phi$, recall that $\Phi_\alpha$ denotes the set of roots $\nu \in \Phi \backslash \{ \alpha \}$ such that $\nu + \alpha$ is a root.

\begin{lemma}\label{lemma:sum:orthogonal}
Let $\alpha$, $\lambda$ be distinct roots in $\Phi$, with $\Phi_\lambda = \emptyset$. Let $\Psi \subset \Pi$ and let $\beta = \sum_{\nu \in \Psi} b_\nu \nu$ be a root in $\Sigma^{+} \backslash \Phi$, where $b_\nu$ is a non-negative integer for each $\nu \in \Psi$. Let us assume that $\lambda+\beta$ is a root and that $ |\lambda|= |\alpha| = |\nu|$ for all $\nu \in \Psi \subset \Pi$. 
\begin{enumerate}[{\rm (i)}]
\item If~\eqref{equation:derivation} holds for $X = \eta_{\alpha, \lambda}$ and any $Y \in \g{g}_\beta$, then $c = - a^2_\lambda |\lambda|^2$.\label{lemma:sum:orthogonal:1}
\item If $\Phi_\alpha = \emptyset$ and $\alpha + \beta$ is not a root, then~\eqref{equation:derivation} holds for any $X \in \R \eta_{\alpha, \lambda}$ and any $Y \in \g{g}_\beta$ if and only if $c = - a^2_\lambda |\lambda|^2$. \label{lemma:sum:orthogonal:2}
\item  If $\Phi_\alpha = \emptyset$ and $\alpha + \beta$ is a root, then~\eqref{equation:derivation} holds for any $X \in \R \eta_{\alpha, \lambda}$ and any $Y \in \g{g}_\beta$ if and only if $c = - a^2_\lambda |\lambda|^2 =- a^2_\alpha |\alpha|^2$.\label{lemma:sum:orthogonal:3}
\end{enumerate}
\end{lemma}

\begin{proof}
Let $Y$ be a vector in $\g{g}_\beta$. Note that $\alpha + \beta$ might be a root. Since $|\nu| = |\lambda|= |\alpha|$ for all $\nu \in \Psi \subset \Pi$, then $\dim \g{g}_\alpha = \dim \g{g}_\lambda = \dim \g{g}_\nu$, and $2 \alpha$, $2 \lambda$ and $2 \nu$ are not roots, for any $\nu \in \Psi$. Using all these considerations together with~\eqref{definition:D}, Proposition~\ref{proposition:ad}~(\ref{proposition:ad:1}) and Proposition~\ref{proposition:rxi:sxi}~(\ref{proposition:rxi:sxi:2}), we get 
\begin{align*}
D[\eta_{\alpha, \lambda}, Y] & = D([\eta_{\alpha, \lambda}, Y]_{\g{g}_{\lambda+ \beta}} + [\eta_{\alpha, \lambda}, Y]_{\g{g}_{\alpha+ \beta}} ) \\
& = \sum_{\nu \in \{ \lambda, \alpha  \}} \left( (l(\beta) + 1) |\lambda|^2 \dim \g{g}_\lambda  - \frac{1}{2} \sum_{\gamma \in \Phi} a^2_\gamma |\gamma|^2 A_{\gamma, \nu + \beta} +c\right) [\eta_{\alpha, \lambda}, Y]_{\g{g}_{\nu+ \beta}}, \\
[ \eta_{\alpha, \lambda},D Y]& = \left( l(\beta) |\lambda|^2 \dim \g{g}_\lambda - \frac{1}{2} \sum_{\gamma \in \Phi} a^2_\gamma |\gamma|^2 A_{\gamma, \beta} +c \right)[\eta_{\alpha, \lambda}, Y].
\end{align*}
Note that $\lambda$ is not in $\Phi_\alpha$ since $\Phi_\lambda = \emptyset$. Using Proposition~\ref{proposition:ad}~(\ref{proposition:ad:2}) and Proposition~\ref{propostion:rs:eta}~(\ref{propostion:rs:eta:1}) with $\Phi_\lambda = \emptyset$, we deduce 
\begin{align*}
[D \eta_{\alpha, \lambda},Y] & =( |\lambda|^2 \dim \g{g}_\lambda  + c)[\eta_{\alpha, \lambda}, Y] -  \frac{1}{2} l_{\alpha, \lambda}  \sum_{\nu \in \Phi_\alpha} a_\lambda a_\nu |\alpha|^2 A_{\alpha, \nu} l_{\alpha, \nu}^{-1} [\eta_{\alpha,\nu}, Y].
\end{align*}
Since $\lambda$, $\beta$, $\lambda + \beta \in \Sigma^{+}$ by assumption, we can select $Y$ in $\g{g}_\beta$ such that $[\eta_{\alpha, \lambda}, Y]_{\g{g}_{\lambda+ \beta}} \neq 0$ by means of Lemma~\ref{lemma:bracket:root:spaces}~(\ref{lemma:bracket:root:spaces:1}). Hence, using~\eqref{equation:derivation} with the projection onto $\g{g}_{\lambda+ \beta}$ of the three equations above, we deduce that $c = - a^2_\lambda |\lambda|^2$. This proves assertion~(\ref{lemma:sum:orthogonal:1}).
	
Let us then assume $\Phi_\alpha = \emptyset$. Thus, the second addend in the right hand side term of $[D \eta_{\alpha, \lambda},Y]$ above vanishes. On the one hand, if $\alpha + \beta$ is not a root, then $[\eta_{\alpha, \lambda}, Y]$ has trivial orthogonal projection onto $\g{g}_{\alpha+\beta} = 0$, and assertion~(\ref{lemma:sum:orthogonal:2}) easily follows.
	
On the other hand, if $\alpha + \beta$ is a root, then interchanging the roles of $\alpha$ and $\lambda$ in assertion~(\ref{lemma:sum:orthogonal:1}) of this lemma, we get that $c = - a^2_\lambda |\lambda|^2 = - a^2_\alpha |\alpha|^2$ is a necessary condition for~\eqref{equation:derivation} to hold for any $X$ in $\R \eta_{\alpha, \lambda}$ and any $Y \in \g{g}_\beta$. From the expressions of $D[\eta_{\alpha, \lambda}, Y]$, $[D\eta_{\alpha, \lambda}, Y]$ and $[\eta_{\alpha, \lambda}, DY]$ calculated above, we deduce that it is also a sufficient condition and this concludes the proof.
\end{proof}

Now, we are in position to prove the following

\begin{proposition}\label{proposition:orthogonal}
Assume that $S$ is a Ricci soliton and let $\alpha$, $\lambda \in \Phi$ be orthogonal roots. Then, one of the following conditions holds:
\begin{enumerate}[{\rm (i)}]
\item $\Phi = \{ \alpha, \lambda \}$, $\Pi$ generates an $A_3$ root system, $a_\alpha = a_\lambda =  1/\sqrt{2}$ and $c = - (1/2) |\lambda|^2$.   \label{proposition:orthogonal:1} 
\item $\Phi = \{ \alpha, \lambda \}$, with $\lambda$ and $\alpha$ of different length. \label{proposition:orthogonal:2}
\end{enumerate}
\end{proposition}
\begin{proof}
For this proof, we will see that if assertion~(\ref{proposition:orthogonal:2}) is not true, then~(\ref{proposition:orthogonal:1}) must hold. If  assertion~(\ref{proposition:orthogonal:2}) is not true, then $|\alpha| = |\lambda|$ or $\Phi$ contains three or more roots. If $\Phi$ contains three or more roots, then two of them will be orthogonal and of the same length. This follows from combining Proposition~\ref{proposition:2:roots} with the fact that there are at most two different possible lengths for the simple roots.

All in all, if assertion~(\ref{proposition:orthogonal:2}) does not hold, we can and will assume that the orthogonal roots $\alpha$, $\lambda \in \Phi$ of the statement have the same length. Also from Proposition~\ref{proposition:2:roots}, we deduce that $\Phi_\lambda = \emptyset$ without loss of generality. Let $\Psi$ be the smallest connected subset of $\Pi$ containing $\alpha$ and $\lambda$. Looking at the different configurations that a Dynkin diagram can adopt (see~\cite[p.~336-340]{Jurgen}) and having in mind the fact that $\Pi$ is irreducible, we deduce that all the roots in $\Psi$ have the same length. Now, take $\beta \in \Psi \cap (\Pi \backslash \Phi)$ connected to $\lambda$ in the Dynkin diagram. It clearly has the same length as $\lambda$ and $\alpha$. 

Since $S$ is a Ricci soliton, then~\eqref{equation:derivation} must hold for any $X$, $Y \in \g{s}$. The roots $\lambda$, $\alpha$ and $\beta$ are under the hypotheses of Lemma~\ref{lemma:sum:orthogonal}, and then we deduce that $c = - a^2_\lambda |\lambda|^2 < 0$. But, since $\rank M \geq 3$, from Corollary~\ref{corollary:c}~(\ref{corollary:c:1}) we get that $\Pi=\{ \alpha, \beta, \lambda \}$ generates an $A_3$ root system. Hence, using Lemma~\ref{lemma:sum:orthogonal}~(\ref{lemma:sum:orthogonal:3}) we deduce that $c = - a^2_\lambda |\lambda|^2 =  - a^2_\alpha |\alpha|^2$. This implies $a_\alpha= a_\lambda$, since $a_\alpha$, $a_\lambda > 0$ (see Lemma~\ref{lemma:subalgebra:s:xi}). Finally, combining this with 
$a^2_\alpha + a^2_\lambda = 1$, which follows from the fact that $\Phi = \{ \alpha, \lambda \}$, one gets assertion~(\ref{proposition:orthogonal:1}).
\end{proof}

The next result shows that the converse to Proposition~\ref{proposition:orthogonal}~(\ref{proposition:orthogonal:2}) is never true.

\begin{proposition}\label{proposition:2:differentlength:orthogonal}
Let $\Phi = \{ \alpha, \lambda\}$, with $\alpha$ and $\lambda$ orthogonal simple roots of different length. Then, $S$ is not a Ricci soliton.
\end{proposition}
\begin{proof}
We will assume that $S$ is a Ricci soliton and we will get a contradiction. Since $\alpha$, $\lambda \in \Pi$ are orthogonal and of different length, we have that $\rank M \geq 3$ and that $\Pi$ cannot generate an $A_3$ root system. Thus, $c = 0$ in~\eqref{definition:D} by virtue of Corollary~\ref{corollary:c}~(\ref{corollary:c:1}). Moreover, $\Pi$ cannot generate a $G_2$ root system, and then we can write $2|\alpha|^2 = |\lambda|^2$ without loss of generality. Hence, $\alpha$ is the unique simple root such that $2\alpha$ might be a root.
	
Let $\beta_1$, $\beta_2 \in \Pi \backslash \Phi$ be roots connected to $\alpha$ and $\lambda$ in the Dynkin diagram, respectively. Note that it may happen $\beta_1 = \beta_2$. Now, we will check~\eqref{equation:derivation} for $X = \eta_{\alpha, \lambda}$ and any unit vector $Y_\nu \in \g{g}_\nu$, with $\nu \in \{ \beta_1, \beta_2 \}$. First, using~\eqref{definition:D} recalling that $c = 0$, Proposition~\ref{proposition:ad}~(\ref{proposition:ad:1}) and Proposition~\ref{proposition:rxi:sxi}~(\ref{proposition:rxi:sxi:2}), we deduce 
\begin{align*}
\nonumber D[\eta_{\alpha, \lambda}, Y_\nu] & =  D([\eta_{\alpha, \lambda}, Y_\nu]_{\g{g}_{\alpha+ \nu}}+[\eta_{\alpha, \lambda}, Y_\nu]_{\g{g}_{\lambda+ \nu}})\\
\nonumber & =\left( |\alpha|^2 (\dim \g{g}_\alpha + 2 \dim \g{g}_{2 \alpha}) + |\nu|^2 \dim \g{g}_\nu - \frac{1}{2}\sum_{\gamma \in \Phi} a^2_\gamma |\gamma|^2 A_{\gamma, \alpha + \nu} \right) [\eta_{\alpha, \lambda}, Y_\nu]_{\g{g}_{\alpha+ \nu}}\\
& \phantom{ = }  +\left( |\lambda|^2 \dim \g{g}_\lambda + |\nu|^2 \dim \g{g}_\nu - \frac{1}{2}\sum_{\gamma \in \Phi} a^2_\gamma |\gamma|^2 A_{\gamma, \lambda + \nu} \right) [\eta_{\alpha, \lambda}, Y_\nu]_{\g{g}_{\lambda+ \nu}},
\end{align*}
and 
\begin{align*}
[\eta_{\alpha, \lambda}, DY_\nu] & = \left(|\nu|^2 \dim \g{g}_\nu - \frac{1}{2}\sum_{\gamma \in \Phi} a^2_\gamma |\gamma|^2 A_{\gamma, \nu}\right)[\eta_{\alpha, \lambda}, Y_\nu].
\end{align*}
Similarly, using Proposition~\ref{proposition:ad}~(\ref{proposition:ad:3}), Proposition~\ref{propostion:rs:eta}~(\ref{propostion:rs:eta:2}) and $A_{\alpha, \lambda} = 0$, we get 
\begin{align*}
[D\eta_{\alpha, \lambda}, Y_\nu] = (a^2_\lambda |\alpha|^2 (\dim \g{g}_\alpha + 2\dim \g{g}_{2\alpha}) + a^2_\alpha |\lambda|^2 \dim \g{g}_\lambda)  [\eta_{\alpha, \lambda}, Y_\nu].
\end{align*} 
Note that $[\eta_{\alpha, \lambda}, Y_\nu] \neq 0$ for each unit $Y_\nu \in \g{g}_\nu$, with $\nu \in \{ \beta_1, \beta_2 \}$, by means of Lemma~\ref{lemma:theta:alpha:new}~(\ref{lemma:theta:alpha:new:3}). On the one hand, we deduce $\dim \g{g}_\alpha  +2 \dim \g{g}_{2\alpha} - 1 - 2\dim \g{g}_\lambda = 0$ if $\nu = \beta_1$, by using equation~\eqref{equation:derivation} and the orthogonal projection onto $\g{g}_{\alpha+ \beta_1}$ of the three equations above. On the other hand, we have $-\dim \g{g}_\alpha  -2 \dim \g{g}_{2\alpha} - 2 + 2\dim \g{g}_\lambda = 0$ if $\nu = \beta_2$, by using equation~\eqref{equation:derivation} and the orthogonal projection onto $\g{g}_{\lambda + \beta_2}$ of the three equations above. Altogether, we obtain that $-3 = 0$, which is a contradiction and finishes the proof.
\end{proof}

Let us summarize the information we have achieved so far in the following

\begin{corollary}\label{corollary:summary}
Assume that $S$ is a Ricci soliton. Then, one of the following conditions holds:
\begin{enumerate}[{\rm (i)}]
\item $\Phi = \{ \alpha \}$ for some $\alpha \in \Pi$.\label{corollary:summary:1}
\item $\Pi$ generates a root system of type $A_3$, $\Phi = \{ \alpha, \lambda \}$ with $A_{\alpha, \lambda} = 0$, $a_\alpha = a_\lambda = 1/\sqrt{2}$ and $c = - (1/2) |\lambda|^2$. \label{corollary:summary:3}
\item $\Phi = \Pi$ generates a root system of type $A_2$. \label{corollary:summary:2}
\item $\Phi$ generates a root subsystem of $\Sigma$ of type $B_2 = C_2 $, $BC_2$ or $G_2$.\label{corollary:summary:4}  
\end{enumerate}
\end{corollary}
\begin{proof}
Let us assume that $\Phi$ contains two or more roots, since otherwise we are in case~(\ref{corollary:summary:1}). On the one hand, if $\Phi$ contains orthogonal roots, then~(\ref{corollary:summary:3}) must hold, as follows from combining Proposition~\ref{proposition:orthogonal} and Proposition~\ref{proposition:2:differentlength:orthogonal}. On the other hand, assume that $\Phi$ does not contain orthogonal roots. Using this and Proposition~\ref{proposition:2:roots}, we get that either $\Pi = \Phi$ generates an $A_2$ root system, which corresponds to case~(\ref{corollary:summary:2}), or assertion~(\ref{corollary:summary:4}) must hold. 
\end{proof}

Cases~(\ref{corollary:summary:1}),~(\ref{corollary:summary:3}) and~(\ref{corollary:summary:2}) of Corollary~\ref{corollary:summary} will be addressed in Section~\ref{section:main:theorem}. In the last part of this section we deal with case~(\ref{corollary:summary:4}) of Corollary~\ref{corollary:summary}, that is, the case in which $\Phi$ is a connected subset of two elements (of different length) of $\Pi$. 

\begin{lemma}\label{lemma:nonzero}
Let $\alpha$, $\lambda \in \Pi$ be connected roots in the Dynkin diagram, with $|\lambda| \geq |\alpha|$ and $\dim \g{g}_\lambda \geq 2$. Then $[[\theta X, Y], Z]$ is a non-zero vector in $\g{g}_\alpha \ominus \R Z$, for any non-zero orthogonal vectors $X$, $Y \in \g{g}_\lambda$ and any non-zero vector $Z \in \g{g}_\alpha$.
\end{lemma}

\begin{proof}
Note that $|\lambda| \geq |\alpha|$ implies $\dim \g{g}_\alpha \geq \dim \g{g}_\lambda$. Since $\alpha - \lambda$ is not a root, using Jacobi identity we get $[[\theta X, Y], Z] = -[[Y, Z], \theta X]$. By virtue of Lemma~\ref{lemma:theta:alpha:new}~(\ref{lemma:theta:alpha:new:3}), we have $W = [Y, Z] \neq 0$. Moreover, using~\eqref{eq:relation:inner}, \eqref{eq:cartan:inner}, Jacobi identity, the fact that $2\lambda + \alpha$ is not a root since $A_{\lambda, \alpha} = -1$, and Lemma~\ref{lemma:berndt:sanmartin}~(\ref{lemma:berndt:sanmartin:i}), we deduce 
\begin{align*}
\langle [W, \theta X], [W, \theta X] \rangle & = \frac{1}{2} \langle [W, \theta X], [W, \theta X] \rangle_{B_\theta}= \frac{1}{2} \langle W, [X, [W, \theta X] ]\rangle_{B_\theta}\\
& = -\frac{1}{2} \langle W,  [W, [\theta X, X] ] \rangle_{B_\theta} 
= \langle X, X \rangle \langle W, [H_\lambda, W] \rangle_{B_\theta} \\
& =  |\lambda|^2 A_{\lambda, \lambda + \alpha} \langle X, X \rangle  \langle W, W \rangle=  |\lambda|^2 \langle X, X \rangle \langle W, W \rangle.
\end{align*} 
This proves $[[\theta X, Y], Z] \neq 0$. Moreover, $[[\theta X, Y], Z] \in \g{g}_\alpha \ominus \R Z$ by combining Lemma~\ref{lemma:root:spaces:k0} with the fact that $[\theta X, Y]$ is in $\g{k}_0$, as follows from Lemma~\ref{lemma:berndt:sanmartin}~(\ref{lemma:berndt:sanmartin:ii}).
\end{proof}

\begin{proposition}\label{proposition:BCr}
Assume that $\Phi$ generates a root subsystem of type $B_2$ or $BC_2$ of $\Sigma$. If $S$ is a Ricci soliton, then $\Phi$ generates a root subsystem of type $B_2$ and $\dim \g{g}_\lambda = 1$, where $\lambda$ is the longest root in $\Phi$.  	
\end{proposition}

\begin{proof}
Put $\Phi =\{\lambda, \alpha \}$ with $|\lambda|^2 = 2|\alpha|^2$, and $\xi = a_\alpha \xi_\alpha + a_\lambda \xi_\lambda$ for the unit vectors $\xi_\alpha \in \g{g}_\alpha$, $\xi_\lambda \in \g{g}_\lambda$, and the positive numbers $a_\alpha$, $a_\lambda$ satisfying $a^2_\alpha+a^2_\lambda = 1$. 
	
Note that if $\Sigma$ is of type $BC_n$ (i.e.\ $2 \alpha$ is a root), $n\geq 2$, then $\dim \g{g}_\lambda > 1$ (see~\cite[p.~340]{Jurgen}). Hence, for the proof of this result we will assume that $\dim \g{g}_\lambda > 1$ and we will get a contradiction with the fact that $S$ is a Ricci soliton. Recall that $|\lambda| \geq |\alpha|$ implies $\dim \g{g}_\alpha \geq \dim \g{g}_\lambda$.

Take a non-zero vector $X_\lambda \in \g{g}_\lambda \ominus \R \xi_\lambda$. Let us define $X_\alpha = [[\theta \xi_\lambda, X_\lambda], \xi_\alpha]$, which is a non-zero vector in $\g{g}_\alpha \ominus \R \xi_\alpha$, as follows from Lemma~\ref{lemma:nonzero}. We will check~\eqref{equation:derivation} for $X = X_\lambda$ and any $Y \in \g{g}_{\lambda + \alpha}$. Since $2\lambda+\alpha$ is not a root, then $[X_\lambda, Y] = 0$. Thus $D[X_\lambda, Y] =0$, and using Proposition~\ref{proposition:ad}~(\ref{proposition:ad:1}) and~Proposition~\ref{proposition:rxi:sxi}~(\ref{proposition:rxi:sxi:2}) we deduce that $ [X_\lambda, D Y] = 0$. 
	
However, using Proposition~\ref{proposition:ad}~(\ref{proposition:ad:1}), Proposition~\ref{proposition:rxi:sxi}~(\ref{proposition:rxi:sxi:4}) and $[X_\lambda, Y] = 0$ we get  
\[
[D X_\lambda, Y] = \frac{1}{2} a_\alpha a_\lambda [[[\theta \xi_\lambda, X_\lambda], \xi_\alpha], Y] =\frac{1}{2} a_\alpha a_\lambda [X_\alpha, Y], 
\]
which is non-zero for a suitable choice of $Y \in \g{g}_{\lambda + \alpha} \subset \g{s}$, as follows from Lemma~\ref{lemma:bracket:root:spaces}~(\ref{lemma:bracket:root:spaces:1}).
\end{proof}

The above result allows us to delete $BC_2$ in case~(\ref{corollary:summary:4}) of Corollary~\ref{corollary:summary}. We state now an auxiliary lemma that will help in order to address case $B_2$ of Corollary~\ref{corollary:summary}~(\ref{corollary:summary:4}).

\begin{lemma}\label{lemma:Bn:Cn:F4}
Assume that $\Phi =\{\lambda, \alpha\}$ generates a root subsystem of $\Sigma$ of type $B_2$, with $|\lambda|^2 = 2|\alpha|^2$ and $\dim \g{g}_\alpha > 1$. Then, equation~\eqref{equation:derivation} holds for any $X \in\g{g}_{\lambda + \alpha}$ and any  $Y \in \g{g}_{\alpha} \ominus \R \xi_{\alpha}$ if and only if $c = - a^2_\alpha |\alpha|^2$. If $S$ is a Ricci soliton, then $c = - a^2_\alpha |\alpha|^2$.
\end{lemma}

\begin{proof}
We have that  $\xi = a_\alpha \xi_\alpha + a_\lambda \xi_\lambda$, for the unit vectors $\xi_\alpha \in \g{g}_\alpha$, $\xi_\lambda \in \g{g}_\lambda$ and the positive numbers $a_\alpha$, $a_\lambda$. Let $X$ be a vector in $\g{g}_{\lambda + \alpha}$ and let $Y$ be a vector in $\g{g}_{\alpha} \ominus \R \xi_{\alpha}$. Then, using~\eqref{definition:D}, Proposition~\ref{proposition:ad}~(\ref{proposition:ad:1}) and Proposition~\ref{proposition:rxi:sxi}~(\ref{proposition:rxi:sxi:2})-(\ref{proposition:rxi:sxi:4}), we deduce 
\begin{align}\label{equation:Bn:Cn:F4}
\nonumber D[X, Y]& = \left(|\lambda|^2 \dim \g{g}_\lambda + 2 |\alpha|^2 \dim \g{g}_\alpha - \frac{1}{2} \sum_{\gamma \in \Phi} a^2_\gamma |\gamma|^2 A_{\gamma, \lambda + 2 \alpha}+c  \right) [X, Y],  \\
[DX, Y]& = \left(|\lambda|^2 \dim \g{g}_\lambda +  |\alpha|^2 \dim \g{g}_\alpha - \frac{1}{2} \sum_{\gamma \in \Phi} a^2_\gamma |\gamma|^2 A_{\gamma, \lambda +  \alpha} +c \right) [X, Y],  \\
\nonumber[X, DY]& = \left(|\alpha|^2 \dim \g{g}_\alpha - \frac{1}{2} a^2_\lambda |\alpha|^2 A_{\alpha, \lambda}+c  \right) [X, Y]  +\frac{1}{2} a_\lambda a_\alpha [X, [[\theta \xi_\alpha, Y], \xi_\lambda]].
\end{align}
Since $2\lambda + \alpha$ is not a root, then $[X, [[\theta \xi_\alpha, Y], \xi_\lambda]] = 0$. If $[X, Y] = 0$, then~\eqref{equation:derivation} holds trivially. However, for a suitable choice of $X \in \g{g}_{\lambda + \alpha}$, we have that $[X, Y] \neq 0$ by virtue of Lemma~\ref{lemma:bracket:root:spaces}~(\ref{lemma:bracket:root:spaces:1}). Hence, using $A_{\alpha, \lambda} = 2 A_{\lambda, \alpha} = -2$ and $2|\alpha|^2 = |\lambda|^2$ in~\eqref{equation:Bn:Cn:F4}, we get that~\eqref{equation:derivation} holds for any $X \in\g{g}_{\lambda + \alpha}$ and any  $Y \in \g{g}_{\alpha} \ominus \R \xi_{\alpha}$ if and only if $c = - a^2_\alpha |\alpha|^2$. When $S$ is a Ricci soliton, \eqref{equation:derivation} must hold for any $X$, $Y \in \g{s}$ and the result follows.
\end{proof}

Let us assume then that $\Phi$ is a $B_2$ or a $G_2$ simple subsystem of $\Pi$. 

\begin{proposition}\label{proposition:B2:G2}
Let $\alpha$, $\lambda \in \Phi$ be connected roots in the Dynkin diagram with $|\lambda|>|\alpha|$. If $S$ is a Ricci soliton, then $\Phi = \Pi$ generates a root system of type $B_2$, $c = a_\alpha^2 |\alpha|^2 (\dim \g{g}_\alpha - 2 \dim \g{g}_\lambda -2)$, with $\dim \g{g}_\lambda=1$ and $\dim \g{g}_\alpha \in \{1,3\}$. If $\dim \g{g}_\alpha =1$, that is, $M = SO^{0}_{2,3}/SO_2 SO_3$, then the converse is true.
\end{proposition}

\begin{proof}
According to Corollary~\ref{corollary:summary} and Proposition~\ref{proposition:BCr}, we can and will assume that $\Phi$ generates a root susbsystem of $\Sigma$ either of type $B_2$ or of type $G_2$. We will check~\eqref{equation:derivation} for $X = \eta_{\alpha, \lambda}$ and any $Y \in \g{g}_{\lambda + \alpha}$. Note that $\lambda + 2\alpha$ is a root, but not $2\lambda + \alpha$. Using~\eqref{definition:D}, Proposition~\ref{proposition:ad}~(\ref{proposition:ad:1})-(\ref{proposition:ad:3}), Proposition~\ref{proposition:rxi:sxi}~(\ref{proposition:rxi:sxi:2}) and Proposition~\ref{propostion:rs:eta}~(\ref{propostion:rs:eta:2}), we get 
\begin{align}\label{proposition:B2:G2:eq:1}
D[\eta_{\alpha, \lambda},Y] & = \left( |\lambda|^2 \dim \g{g}_\lambda + 2|\alpha|^2 \dim \g{g}_\alpha - \frac{1}{2} \sum_{\gamma \in \Phi } a^2_\gamma |\gamma|^2 A_{\gamma, \lambda+ 2\alpha}  +c  \right)[\eta_{\alpha, \lambda},Y], \nonumber \\
[D\eta_{\alpha, \lambda}, Y] & = \left(a^2_\lambda |\alpha|^2 \dim \g{g}_\alpha +  a^2_\alpha |\lambda|^2 \dim \g{g}_\lambda - \frac{1}{2} |\alpha|^2 A_{\alpha, \lambda}  +c  \right)[\eta_{\alpha, \lambda},Y],\\
[\eta_{\alpha, \lambda},DY] & = \left( |\lambda|^2 \dim \g{g}_\lambda + |\alpha|^2 \dim \g{g}_\alpha - \frac{1}{2} \sum_{\gamma \in \Phi } a^2_\gamma |\gamma|^2 A_{\gamma, \lambda+ \alpha}  +c  \right)[\eta_{\alpha, \lambda},Y].\nonumber
\end{align}
It is clear that~\eqref{equation:derivation} is satisfied whenever $[\eta_{\alpha, \lambda},Y]=0$. However, let us assume that $[\eta_{\alpha, \lambda}, Y] \neq 0$, which is true for a suitable choice of $Y \in \g{g}_{\lambda + \alpha}$, as follows from Lemma~\ref{lemma:bracket:root:spaces}~(\ref{lemma:bracket:root:spaces:1}).

First, let us assume that $\Pi = \Phi$ generates a root system $\Sigma$ of type $G_2$. Then $A_{\alpha, \lambda} = 3 A_{\lambda, \alpha } = -3$ and $\dim \g{g}_\alpha = \dim \g{g}_\lambda$ (see~\cite[p.~339]{Jurgen}). Using this together with~\eqref{equation:derivation},~\eqref{proposition:B2:G2:eq:1} and Corollary~\ref{corollary:c}~(\ref{corollary:c:2}), we deduce that $0 = c = a_\alpha^2 |\alpha|^2( -2 \dim \g{g}_\alpha -5/2)$, which is a contradiction. 

Now, let us assume that $\Phi$ generates a root subsystem of $\Sigma$ of type $B_2$. Therefore, we have that $A_{\alpha, \lambda} = 2 A_{\lambda, \alpha } = -2$. Since $S$ is a Ricci soliton, from Proposition~\ref{proposition:BCr} we get $\dim \g{g}_\lambda = 1$. Taking into account these considerations and combining~\eqref{equation:derivation} with~\eqref{proposition:B2:G2:eq:1}, we deduce that~\eqref{equation:derivation} holds for $X = \eta_{\alpha, \lambda}$ and any $Y \in \g{g}_{\lambda + \alpha}$ if and only if
\begin{equation}\label{proposition:B2:G2:eq:2}
c = a_\alpha^2 |\alpha|^2 (\dim \g{g}_\alpha - 2 \dim \g{g}_\lambda -2) = a_\alpha^2 |\alpha|^2 (\dim \g{g}_\alpha - 4).
\end{equation}
First, if $\rank M \geq 3$, we get $c = 0$ by virtue of Corollary~\ref{corollary:c}~(\ref{corollary:c:1}). Thus $\dim \g{g}_\alpha = 4$ by means of~\eqref{proposition:B2:G2:eq:2} and $c = -a_\alpha^2 |\alpha|^2$ from Lemma~\ref{lemma:Bn:Cn:F4}, which is a contradiction. Hence, $\Phi = \Pi$ generates the root system $\Sigma$ of type $B_2$.

Moreover, if $\dim \g{g}_\alpha >1$, combining Lemma~\ref{lemma:Bn:Cn:F4} and~\eqref{proposition:B2:G2:eq:2} we deduce that $\dim \g{g}_\alpha = 3$ and $c = -a_\alpha^2 |\alpha|^2$. Otherwise, the only restriction on $c$ is~\eqref{proposition:B2:G2:eq:2}.

Finally, let us see that $S$ is a Ricci soliton when $\dim \g{g}_\alpha = 1$, that is, when $M = SO^{0}_{2,3}/SO_2 SO_3$. In this case, each of the subspaces of decomposition $\g{s} = \R \eta_{\alpha, \lambda} \oplus \g{g}_{\lambda + \alpha} \oplus \g{g}_{\lambda + 2\alpha}$ is an eigenspace of the endomorphism $D$, as follows from Proposition~\ref{proposition:ad}~(\ref{proposition:ad:1})-(\ref{proposition:ad:3}), Proposition~\ref{proposition:rxi:sxi}~(\ref{proposition:rxi:sxi:2}) and Proposition~\ref{propostion:rs:eta}~(\ref{propostion:rs:eta:2}). Hence, we just need to check~\eqref{equation:derivation} for any $X \in \R \eta_{\alpha, \lambda}$ and any $Y \in \g{g}_{\lambda + \alpha}$, since the rest of the brackets are zero. But this is true for $c$ as in ~\eqref{proposition:B2:G2:eq:2} by means of~\eqref{proposition:B2:G2:eq:1}. This concludes the proof.
\end{proof}

The next result shows that $S$ is not a Ricci soliton,  under the conditions of Proposition~\ref{proposition:B2:G2}, when $\dim \g{g}_\alpha = 3$.

\begin{proposition}\label{proposition:B2}
Let $\alpha$, $\lambda \in \Phi$ be connected roots in the Dynkin diagram with $|\lambda| > |\alpha|$. Then, $S$ is a Ricci soliton if and only if $M = SO^{0}_{2,3}/SO_2 SO_3$.
\end{proposition}

\begin{proof}
As follows from Proposition~\ref{proposition:B2:G2}, it suffices to check $S$ is not a Ricci soliton if $M$ is the symmetric space of non-compact type and rank two $SO^{0}_{2,5}/SO_2 SO_5$, 

Put $\Phi =\Pi =\{ \lambda, \alpha\}$, where $|\lambda|^2 = 2 |\alpha|^2$. Note that $A_{\alpha, \lambda} = 2 A_{\lambda, \alpha} = -2$ and that $\dim \g{g}_\alpha =3$. Take a unit vector $Y \in \g{g}_\alpha \ominus \R \xi_\alpha$. Then $[\eta_{\alpha, \lambda}, Y] \neq 0$ by virtue of Lemma~\ref{lemma:theta:alpha:new}~(\ref{lemma:theta:alpha:new:3}). We have that $[[\theta \xi_\alpha, Y], \xi_\lambda] = 0$, as follows from combining $\dim \g{g}_\lambda = 1$ with  Lemma~\ref{lemma:root:spaces:k0} and the fact that $[\theta \xi_\alpha, Y]$ is in $\g{k}_0$, by virtue of Lemma~\ref{lemma:berndt:sanmartin}~(\ref{lemma:berndt:sanmartin:ii}). Using this, Proposition~\ref{proposition:ad}~(\ref{proposition:ad:1})-(\ref{proposition:ad:3}), Proposition~\ref{proposition:rxi:sxi}~(\ref{proposition:rxi:sxi:2})-(\ref{proposition:rxi:sxi:4}), Proposition~\ref{propostion:rs:eta}~(\ref{propostion:rs:eta:2}), we deduce
\begin{align*}
D[\eta_{\alpha, \lambda}, Y] & =\left( |\lambda|^2 \dim \g{g}_\lambda + |\alpha|^2 \dim \g{g}_\alpha - \frac{1}{2} \sum_{\gamma \in \Phi} a^2_\gamma |\gamma|^2 A_{\gamma, \lambda + \alpha} +c \right) [\eta_{\alpha, \lambda}, Y],\\
[D\eta_{\alpha, \lambda}, Y] & =\left(a^2_\lambda |\alpha|^2 \dim \g{g}_\alpha + a^2_\alpha  |\lambda|^2 \dim \g{g}_\lambda  - \frac{1}{2} |\alpha|^2 A_{\alpha, \lambda} +c \right) [\eta_{\alpha, \lambda}, Y],\\
[\eta_{\alpha, \lambda}, DY] & =\left(|\alpha|^2 \dim \g{g}_\alpha - \frac{1}{2} a^2_\lambda |\alpha|^2 A_{\alpha, \lambda} +c \right) [\eta_{\alpha, \lambda}, Y]. 
\end{align*}
Recalling that $\dim \g{g}_\alpha = 3$, $\dim \g{g}_\lambda = 1$, we deduce that if~\eqref{equation:derivation} holds for $X=\eta_{\alpha, \lambda}$ and a unit vector $Y \in \g{g}_\alpha \ominus \R \xi_\alpha$, then $c = -3 a^2_\lambda|\alpha|^2 - |\alpha|^2$. Since $0 < a^2_\lambda, a^2_\alpha < 1$, we get a contradiction with $c = - a^2_\alpha |\alpha|^2$ (Proposition~\ref{proposition:B2:G2}) and the result follows.
\end{proof}

We finish this section with a corollary gathering all the information that we have so far and which follows directly from combining Corollary~\ref{corollary:summary} with Proposition~\ref{proposition:B2}.

\begin{corollary}\label{corollary:summary2}
Assume that $S$ is a Ricci soliton. Then, one of the following conditions holds:
\begin{enumerate}[{\rm (i)}]
\item $\Phi = \{ \alpha \}$ for some $\alpha \in \Pi$.\label{corollary:summary2:1}
\item $\Phi = \Pi$ generates a root system of type $A_2$.\label{corollary:summary2:2}
\item $\Pi$ generates a root system of type $A_3$, $\Phi = \{ \alpha, \lambda \}$ with $A_{\alpha, \lambda} = 0$, $a_\alpha = a_\lambda = 1/\sqrt{2}$ and $c = - (1/2) |\lambda|^2$. \label{corollary:summary2:3}
\item $M = SO^{0}_{2,3}/SO_2 SO_3$ and $\Phi = \Pi$. \label{corollary:summary2:4}  
\end{enumerate}
Moreover, if item~(\ref{corollary:summary2:4}) holds, then $S$ is a Ricci soliton.
\end{corollary}
\section{Proof of the Main Theorem}\label{section:main:theorem}
In this section, we finish the classification of codimension one Lie subgroups of nilpotent Iwasawa groups $N$ which are Ricci solitons when considered with the induced metric.

In order to do so, we examine the different possibilities obtained in Corollary~\ref{corollary:summary2}. After an auxiliary result, in Proposition~\ref{proposition:1:root}  we achieve the classification of codimension one Ricci soliton Lie subgroups of $N$ under the conditions of item~(\ref{corollary:summary2:1}) of Corollary~\ref{corollary:summary2}. After that, we focus on cases~(\ref{corollary:summary2:2}) and~(\ref{corollary:summary2:3}) of Corollary~\ref{corollary:summary2} in Proposition~\ref{proposition:A2} and Proposition~\ref{proposition:A3}, respectively. In the last part of the section, we conclude the proof of the Main Theorem.

\begin{lemma}\label{lemma:orthogonal}
Let $\alpha \in \Phi$ be orthogonal to any other root in $\Phi$ and satisfying $\dim \g{g}_\alpha >1$. Let $\lambda \in \Sigma^{+} \backslash \Phi$ be such that $\lambda + \alpha$ is a root. 
\begin{enumerate}[{\rm (i)}]
\item Assume that $2 \alpha$ is not a root. Then, equation~\eqref{equation:derivation} holds for any $X \in \g{g}_\alpha \ominus \R \xi_\alpha$ and any $Y \in \g{g}_\lambda$ if and only if $c = -a^2_\alpha |\alpha|^2$. \label{lemma:orthogonal:1}
\item Assume that $2 \alpha$ is a root, $\lambda \in \Pi$ and $A_{\lambda, \alpha} = -1$. If equation~\eqref{equation:derivation} holds for any $X \in \g{g}_\alpha \ominus \R \xi_\alpha$ and any $Y \in \g{g}_\lambda$, then $c < 0$. \label{lemma:orthogonal:2}
\end{enumerate}
\end{lemma}
\begin{proof}
Put $\lambda = \sum_{\beta \in \Pi} b_\beta \beta$ and $k_\mu = \dim \g{g}_\mu + 2 \dim \g{g}_{2\mu}$ for each $\mu \in \Pi$. We will check~\eqref{equation:derivation} for $X \in \g{g}_\alpha \ominus \R \xi_\alpha$ and any $Y \in \g{g}_\lambda$. Using Proposition~\ref{proposition:ad}~(\ref{proposition:ad:1}) and Proposition~\ref{proposition:rxi:sxi}~(\ref{proposition:1:root:2}), we deduce 
\begin{align}\label{lemma:orthogonal:eq1}
\nonumber D[X,Y] &= \left(|\alpha|^2 k_\alpha + \sum_{\beta \in \Pi} b_\beta |\beta|^2 k_\beta - \frac{1}{2} \sum_{\gamma \in \Phi} a^2_\gamma |\gamma|^2 A_{\gamma, \lambda + \alpha} +c\right) [X, Y],  \\
& \\
\nonumber[X,DY] & = \left(\sum_{\beta \in \Pi} b_\beta |\beta|^2 k_\beta - \frac{1}{2} \sum_{\gamma \in \Phi} a^2_\gamma |\gamma|^2 A_{\gamma, \lambda } + c\right)[X,Y].
\end{align}
Now, we will simultaneously calculate $[DX, Y]$ when $2 \alpha$ is a root and when it is not. In order to do so, let $\varepsilon =1$ if $2 \alpha$ is a root and $\varepsilon =0$ otherwise. Using Proposition~\ref{proposition:ad}~(\ref{proposition:ad:1}), Proposition~\ref{proposition:rxi:sxi}~(\ref{proposition:1:root:1}) when $2\alpha$ is a root (together with $[X, \theta \xi] \in \g{k}_0$, as follows from Lemma~\ref{lemma:berndt:sanmartin}~(\ref{lemma:berndt:sanmartin:ii}) and the fact that $\alpha$ is orthogonal to any root in $\Phi \backslash \{\alpha\}$) or Proposition~\ref{proposition:rxi:sxi}~(\ref{proposition:1:root:3}) when it is not, we get
\begin{align}\label{lemma:orthogonal:eq2}
[DX,Y] & = \left(|\alpha|^2 k_\alpha + c \right)[X,Y]-\varepsilon\frac{a^2_{\alpha}}{2}[[[X, \xi_\alpha], \theta \xi_\alpha], Y].
\end{align}
Let us assume first that $2 \alpha$ is not a root, that is, $\varepsilon = 0$ in~\eqref{lemma:orthogonal:eq2}. If $[X,Y]=0$, then~\eqref{equation:derivation} holds trivially. However, take $X \in \g{g}_\alpha \ominus \R \xi_\alpha$ and $Y \in \g{g}_\lambda$ such that $[X,Y] \neq 0$, which is possible by virtue of Lemma~\ref{lemma:bracket:root:spaces}~(\ref{lemma:bracket:root:spaces:1}). Now~\eqref{equation:derivation} holds if and only if $c  =- a^2_\alpha |\alpha|^2$. This proves assertion~(\ref{lemma:orthogonal:1}).

Finally, let us assume that $2 \alpha$ is a root, that is, $\varepsilon = 1$ in~\eqref{lemma:orthogonal:eq2}. As before, take $X \in \g{g}_\alpha \ominus \R \xi_\alpha$ and $Y \in \g{g}_\lambda$ such that $[X,Y] \neq 0$. We will consider the $[X,Y]$-component of $[[[X, \xi_\alpha], \theta \xi_\alpha], Y]$. In order to do so, using~\eqref{eq:relation:inner}, \eqref{eq:cartan:inner}, Jacobi identity, the fact that $-\lambda + \alpha$ is not root since $\lambda$ is simple in assertion~(\ref{lemma:orthogonal:2}), Lemma~\ref{lemma:berndt:sanmartin}~(\ref{lemma:berndt:sanmartin:i}) and $A_{\lambda, \alpha} = -1$, we get
\begin{align}\label{lemma:orthogonal:eq3}
\nonumber\langle [[[X, \xi_\alpha], \theta \xi_\alpha], Y], [X,Y] \rangle &= (1/2) \langle [[[X, \xi_\alpha], \theta \xi_\alpha], Y], [X,Y] \rangle_{B_\theta}\\
\nonumber&= (1/2) \langle [[X, \xi_\alpha], \theta \xi_\alpha], [\theta Y, [X,Y]] \rangle_{B_\theta}\\
&=- (1/2)  \langle [[X, \xi_\alpha], \theta \xi_\alpha], [X, [Y, \theta Y]] \rangle_{B_\theta} \\
\nonumber & = - (|\lambda|^2 /2)  \langle Y, Y \rangle A_{\lambda, \alpha}  \langle [[X, \xi_\alpha], \theta \xi_\alpha], X \rangle_{B_\theta}\\
\nonumber& =-|\lambda|^2 \langle Y, Y \rangle \langle [X, \xi_\alpha], [X, \xi_\alpha] \rangle.
\end{align}
As before, we can take $X \in \g{g}_\alpha \ominus \R \xi_\alpha$ and $Y \in \g{g}_\lambda$ such that $[X,Y] \neq 0$. Consider~\eqref{lemma:orthogonal:eq1} and the $[X,Y]$-component of~\eqref{lemma:orthogonal:eq2} by using~\eqref{lemma:orthogonal:eq3}. If~\eqref{equation:derivation} holds for these selected $X$ and $Y$, then $c = -a^2_\alpha |\alpha|^2 - (a^2_{\alpha}/2)|\lambda|^2 \langle [X,Y],  [X,Y] \rangle^{-1}\langle Y, Y \rangle \langle [X, \xi_\alpha], [X, \xi_\alpha] \rangle <0$ and the result follows.
\end{proof}

Firstly, we will obtain the classification of codimension one Ricci solitons $S$ of $N$ when $\Phi = \{ \alpha \}$ for some $\alpha \in \Pi$ or, equivalently, when the unit normal vector field $\xi$ to $S$ in $N$ belongs to the simple root space $\g{g}_{\alpha}$.

\begin{proposition}\label{proposition:1:root}
Let $S$ be a codimension one Lie subgroup of $N$, with $\xi \in \g{g}_\alpha$ for some $\alpha \in \Pi$. Then, $S$ is a Ricci soliton if and only if one of the following conditions holds:
\begin{enumerate}[{\rm (i)}]
\item $M$ is hyperbolic space $\R H^{n}$, $\C H^n$, $n \geq 2$, or a hyperbolic plane $\mathbb{H}H^2$ or $\mathbb{O} H^2$.\label{proposition:1:root:1}
\item $\rank M \geq 2$ and $\g{g}_\alpha = \R \xi$ $(\dim \g{g}_\alpha = 1)$. \label{proposition:1:root:2}
\item $M$ is $SL_3 (\mathbb{C})/SU_3$, $SL_3(\mathbb{H})/Sp_3$, or $E^{-26}_6/F_4$. \label{proposition:1:root:3}
\item $M$ is $SO_5 (\C)/SO_5$ or $SO^{0}_{2,2+n}/SO_2 SO_{2+n}$, $n \geq 2$, with $\alpha$ the shortest simple root.\label{proposition:1:root:4}
\end{enumerate}
\end{proposition}

\begin{remark}
The symmetric spaces $SL_3 (\R)/SO_3$ and $SO^{0}_{2,3}/SO_2 SO_{3}$ do not appear explicitly in this result. This is because $\dim \g{g}_\alpha =1$ in both cases and, thus, the corresponding Ricci solitons are already described in~(\ref{proposition:1:root:2}).
\end{remark}

\begin{proof}
First, let us focus on rank one symmetric spaces. If $M = \R H^n$, with $n \geq 2$, then any codimension one Lie subgroup $S$ of $N$ is a Ricci soliton. This follows from the fact that $\g{s} = \g{g}_\alpha \ominus \R \xi$ is abelian and that $DX$ is proportional to $X$ for any $X \in \g{s}$, by means of Proposition~\ref{proposition:ad}~(\ref{proposition:ad:1}) and Proposition~\ref{proposition:rxi:sxi}~(\ref{proposition:rxi:sxi:3}). For the rest of the cases in~(\ref{proposition:1:root:1}), $S$ is a codimension one Lie subgroup of a generalized Heisenberg group. The classification of such $S$ whose induced metric is a Ricci soliton was achieved in~\cite[Theorem B]{DST20}. This concludes the study of rank one symmetric spaces. 

Hence, from now on in this proof we will assume that $\rank M \geq 2$. Let us justify that the rest of the examples in the statement are indeed Ricci solitons. If $\dim \g{g}_\alpha = 1$, then $\g{s} = \bigoplus_{\lambda \in \Sigma^{+} \backslash \Phi} \g{g}_\lambda$ and~\eqref{equation:derivation} holds for any $X$, $Y \in \g{s}$ by means of Lemma~\ref{lemma:c}. Finally, we will have that $\g{s} = (\g{g}_{\alpha} \ominus \R \xi_\alpha) \oplus \g{g}_\lambda \oplus \g{g}_{\lambda + \alpha}$ if~(\ref{proposition:1:root:3}) holds, and $\g{s} = (\g{g}_{\alpha} \ominus \R \xi_\alpha) \oplus \g{g}_\lambda \oplus \g{g}_{\lambda + \alpha} \oplus \g{g}_{\lambda + 2\alpha}$ if~(\ref{proposition:1:root:4}) does. Note that $D X$ is proportional to $X$ for any $X$ in one of the subspaces of the previous decompositions, by virtue of Proposition~\ref{proposition:ad}~(\ref{proposition:ad:1}) and Proposition~\ref{proposition:rxi:sxi}~(\ref{proposition:rxi:sxi:2})-(\ref{proposition:rxi:sxi:3}). Hence, we just need to check~\eqref{equation:derivation} for vectors $X$, $Y \in \g{s}$ in subspaces of the previous decompositions such that $[X,Y]$ is a non-zero vector. Note that the only brackets that might be non-zero are of the form $[X_\alpha, Y_\nu]$, where $X_\alpha$ is in $\g{g}_{\alpha} \ominus \R \xi_\alpha$ and $Y_\nu$ is in $\g{g}_\nu$, with $\nu = \lambda$ if~(\ref{proposition:1:root:3}) holds, and $\nu \in \{ \lambda, \lambda + \alpha\}$ if~(\ref{proposition:1:root:4}) holds. In both cases, the roots $\alpha$ and $\nu$ are under the hypotheses of Lemma~\ref{lemma:orthogonal}~(\ref{lemma:orthogonal:1}) and then equation~\eqref{equation:derivation} holds with $c = -|\alpha|^2$ ($a_\alpha=1$ since $\Phi = \{ \alpha\}$) for any $X_\alpha$ and any $Y_\nu$ chosen as above. We have seen that all the examples in the statement are indeed Ricci solitons.

Conversely, let $S$ be a Ricci soliton with $\Phi = \{ \alpha \}$ for some $\alpha \in \Pi$, and let us check that it is in the list of the statement. Recall that we can consider $\rank M \geq 2$. If $\dim \g{g}_\alpha =1$, we are led to item~(\ref{proposition:1:root:2}). 

Hence, we can and will assume that $\rank M \geq 2$ and $\dim \g{g}_\alpha >1$. Since $\Pi$ is irreducible, take a root $\lambda \in \Pi$ connected to $\alpha$ in the Dynkin diagram. Then, $\alpha$ and $\lambda$ are under the hypotheses of Lemma~\ref{lemma:orthogonal} and $c < 0$. Therefore, $\Pi$ generates a root system of type $A_2$ or $B_2$, since otherwise $c = 0$ by means of Corollary~\ref{corollary:c}. 

First, if $\Pi$ generates an $A_2$ type root system, then we are in case~(\ref{proposition:1:root:3}) as follows from~\cite[pp.~336-337]{Jurgen}, taking into account that the case $M = SL_3 (\R) / SO_3$ satisfies $\dim \g{g}_{\alpha} =1$ and it is contained in item~(\ref{proposition:1:root:2}). 

Finally, let $\Pi =\{\alpha, \lambda\}$ generate a $B_2$ type root system. If $\alpha \in \Phi$ is the shortest root, then this corresponds to~(\ref{proposition:1:root:4}) as follows from~\cite[p.~337]{Jurgen}, since the case $M = SO^{0}_{2,3} / SO_2 SO_3$ satisfies $\dim \g{g}_{\alpha} =1$ and it is contained in item~(\ref{proposition:1:root:2}). Otherwise, if $\alpha \in \Phi$ were the longest root, the sum of the roots $\lambda$, $\lambda + \alpha \in\Sigma^{+} \backslash \Phi$ would be a root, and using Lemma~\ref{lemma:c}~(\ref{lemma:c:2}) we would get $c=0$, which is a contradiction. 
\end{proof}

Now, we will focus on cases~(\ref{corollary:summary2:2}) and~(\ref{corollary:summary2:3}) of Corollary~\ref{corollary:summary2}. We introduce an auxiliary result in order to address the $A_2$ case.

\begin{lemma}\label{lemma:brackets}
Let $\alpha$, $\lambda \in \Pi$ be connected roots in the Dynkin diagram of the same length. Let $X$, $Y$ be orthogonal vectors in $\g{g}_\alpha$ and let $W$ be a vector in $\g{g}_\lambda$. Then:
\begin{enumerate}[{\rm (i)}]
\item $[[[\theta X, Y], W], X] =  |\alpha|^2 \langle X, X \rangle [Y, W]$. \label{lemma:brackets:1}
\item $[[[\theta X, Y], W], T] =[Y, [[\theta X, T], W]]$, for any $T \in \g{g}_\alpha$ orthogonal to $X$ and $Y$. \label{lemma:brackets:2}
\end{enumerate}
\end{lemma}

\begin{proof}
Let $X$, $Y$, $T$ be mutually orthogonal vectors in $\g{g}_{\alpha}$. Note that the vectors $[\theta X, Y]$, $[\theta Y, T]$ and $[\theta T ,X]$ are in $\g{k}_0$ by means of Lemma~\ref{lemma:berndt:sanmartin}~(\ref{lemma:berndt:sanmartin:ii}). Moreover, since $\alpha$ and $\lambda$ have the same length, we deduce that $2\alpha$ cannot be a root. Taking these considerations into account and using three times the Jacobi identity and $\theta_ {\rvert_{\g{k}}} = \id_{\g{k}}$, we deduce 
\begin{align}\label{quaternionic:k0}
[[\theta X, Y], T] & =   [[X, \theta Y], T] = - [[\theta Y, T], X]  = - [[Y, \theta T], X]  = [[\theta T, X], Y] \nonumber \\
&  =  [[T, \theta X], Y]  = -[[\theta X, Y], T],
\end{align}
which proves that $[[\theta X, Y], T] = 0$ if $X$, $Y$, $T$ are mutually orthogonal vectors in $\g{g}_\alpha$. Moreover, using Jacobi identity and Lemma~\ref{lemma:berndt:sanmartin}~(\ref{lemma:berndt:sanmartin:i}), we deduce  
\begin{equation}\label{lemma:brackets:eq1}
[[\theta X, Y ], X] = -[[X, \theta X], Y] = 2 |\alpha|^2 \langle X, X \rangle Y.
\end{equation}
Now, let $Z$ and $W$ be elements in $\g{g}_\alpha$ and $\g{g}_\lambda$, respectively. Then, using the Jacobi identity three times and taking into account that neither $\lambda- \alpha$ nor $2\alpha +\lambda$ are roots, we have 
\begin{align}\label{lemma:brackets:eq2}
[[[\theta X, Y], W], Z] & = -[[W,Z],[\theta X, Y]] - [[Z, [\theta X, Y]],W] \nonumber\\
&=[Y,[[W,Z], \theta X]] + [\theta X, [Y,[W,Z]]]  -[[Z, [\theta X, Y]],W]\nonumber \\
& =-[Y,[[\theta X,W],Z]] - [Y,[[Z, \theta X], W]]   -[[Z, [\theta X, Y]],W]\nonumber\\
& = [Y,[[\theta X, Z], W]] +[[[\theta X, Y],Z],W].
\end{align}
In particular, if $Z = X$, using Lemma~\ref{lemma:berndt:sanmartin}~(\ref{lemma:berndt:sanmartin:i}), ~\eqref{lemma:brackets:eq1} and $A_{\alpha, \lambda} = -1$ in~\eqref{lemma:brackets:eq2}, we deduce 
\begin{align*}
[[[\theta X, Y], W], X] & =[Y,[[\theta X, X], W]]  + [[[\theta X, Y],X],W]\\
& =  \langle X, X \rangle (2[Y,[H_\alpha, W]]+  2|\alpha|^2 [Y,W])\\
&=\langle X, X \rangle (|\alpha|^2 A_{\alpha, \lambda} [Y, W] + 2 |\alpha|^2 [Y,W]) = |\alpha|^2\langle X, X \rangle [Y, W], 
\end{align*}
which proves assertion~(\ref{lemma:brackets:1}). Finally, let $Z$ be orthogonal to $X$ and $Y$. Then, using~\eqref{quaternionic:k0} in~\eqref{lemma:brackets:eq2}, we deduce that $[[[\theta X, Y], W], Z] = [Y,[[\theta X, Z], W]]$.
\end{proof}

Now, we obtain the classification result when $\Phi = \Pi$ is an $A_2$ simple system.

\begin{proposition}\label{proposition:A2}
Let $M \cong AN$ be an irreducible symmetric space of non-compact type whose root system is of type $A_2$ and let $S$ be a codimension one Lie subgroup of $N$. Then, $S$ is a Ricci soliton if and only if one of the following conditions holds:
\begin{enumerate}[{\rm (i)}]
\item $M$ is $SL_3 (\R)/SO_3$ or  $SL_3 (\C)/SU_3$.\label{proposition:A2:1}
\item $M$ is $SL_3 (\mathbb{H})/Sp_3$ or  $E^{-26}_6/F_4$ and $\xi \in \g{g}_\alpha$ for some $\alpha \in \Pi$.\label{proposition:A2:2}
\end{enumerate} 
\end{proposition}
\begin{proof}
Let $\Pi = \{\alpha, \lambda\}$ be the set of simple roots of the root system $\Sigma$ of $M$. Note that $A_{\alpha, \lambda} = A_{\lambda, \alpha} = -1$, $|\nu|=|\mu|$ and $\dim \g{g}_{\nu} = \dim \g{g}_{\mu}$ for any $\nu$, $\mu \in \Sigma^{+}$. 

First, we will see that the examples of the statement are Ricci solitons. If $\xi \in \g{g}_\alpha$ for some $\alpha \in \Pi$, equivalently if $\Phi = \{ \alpha \}$, then it follows from Proposition~\ref{proposition:1:root}~(\ref{proposition:1:root:2})-(\ref{proposition:1:root:3}). Let us then assume that $\Phi = \Pi$ and that $M$ is either $SL_3 (\R)/SO_3$ or  $SL_3 (\C)/SU_3$. Hence, we have that $\dim \g{g}_\nu \leq 2$ for any $\nu \in \Sigma^{+}$, $\xi = a_\alpha \xi_\alpha + a_\lambda \xi_\lambda$, for the unit vectors $\xi_\alpha \in \g{g}_\alpha$, $\xi_\lambda \in \g{g}_\lambda$, and positive numbers $a_\alpha$, $a_\lambda$ such that $a^2_\alpha +a^2_\lambda = 1$. 

We will check~\eqref{equation:derivation} for vectors $X$, $Y \in \g{s}$ in the subspaces of the orthogonal decomposition
\begin{equation*}
\g{s} = \R \eta_{\alpha, \lambda} \oplus (\g{g}_{\alpha} \ominus \R \xi_\alpha) \oplus (\g{g}_{\lambda} \ominus \R \xi_\lambda) \oplus \g{g}_{\alpha +\lambda}. 
\end{equation*} 
By virtue of definition~\eqref{definition:D} and Proposition~\ref{proposition:ad}~(\ref{proposition:ad:1})-(\ref{proposition:ad:2}), Proposition~\ref{proposition:rxi:sxi}~(\ref{proposition:rxi:sxi:2}) and Proposition~\ref{propostion:rs:eta}~(\ref{propostion:rs:eta:2}), we deduce that $D$ leaves invariant the abelian subalgebra $\R \eta_{\alpha, \lambda} \oplus \g{g}_{\alpha +\lambda}$ of $\g{s}$. Using this we deduce that~\eqref{equation:derivation} holds, with any real number $c$, for any $X$, $Y \in \R \eta_{\alpha, \lambda} \oplus \g{g}_{\alpha +\lambda}$.

Let $X$ be a vector in $(\g{g}_{\alpha} \ominus \R \xi_\alpha) \oplus (\g{g}_{\lambda} \ominus \R \xi_\lambda)$ and $Y$ in $\g{g}_{\lambda + \alpha}$. We have that $D X \in \g{g}_\alpha \oplus \g{g}_\lambda$ and that $DY$ is proportional to $Y$, by means of Proposition~\ref{proposition:ad}~(\ref{proposition:ad:1}) and Proposition~\ref{proposition:rxi:sxi}~(\ref{proposition:rxi:sxi:2})-(\ref{proposition:rxi:sxi:4}). Thus, we deduce that $D[X, Y] = [DX, Y]  = [X, DY] = 0$. Therefore~\eqref{equation:derivation} holds, with any $c$, for any $X \in (\g{g}_{\alpha} \ominus \R \xi_\alpha) \oplus (\g{g}_{\lambda} \ominus \R \xi_\lambda)$ and any $Y \in \g{g}_{\alpha +\lambda}$.

Take a vector $X_\nu \in (\g{g}_{\nu} \ominus \R \xi_\nu)$, for each $\nu \in \{ \lambda, \alpha \}$. Since $\dim \g{g}_{\alpha} \leq 2$, in order to show that all the examples in the statement are Ricci solitons, it suffices to check that~\eqref{equation:derivation} holds for any
\[
(X, Y) \in \{ (X_\nu, X_\nu), (X_\alpha, X_\lambda), (X_\nu, \eta_{\alpha, \lambda}) \},
\]
for each $\nu \in \{ \lambda, \alpha \}$. The case $(X, Y) = (X_\nu, X_\nu)$ is trivial by using the anticommutativity of the Lie bracket product, for $\nu \in \{ \alpha, \lambda \}$. Let us see that in the other cases~\eqref{equation:derivation} holds for $c = -|\alpha|^2$. Note that $2\alpha$ and $2\lambda$ are not roots. Using this, Proposition~\ref{proposition:ad}~(\ref{proposition:ad:1}) and Proposition~\ref{proposition:rxi:sxi}~(\ref{proposition:rxi:sxi:2})-(\ref{proposition:rxi:sxi:4}), we deduce
\begin{align*}
D[X_\alpha, X_\lambda] & = \left( |\alpha|^2 \dim \g{g}_\alpha + |\lambda|^2 \dim \g{g}_\lambda - \frac{1}{2}a^2_\alpha |\alpha|^2   - \frac{1}{2}a^2_\lambda |\lambda|^2 +c\right)[X_\alpha, X_\lambda],\\
[D X_\alpha, X_\lambda] & = \left(|\alpha|^2 \dim \g{g}_\alpha + \frac{1}{2} a^2_\lambda |\alpha|^2  + c \right)[X_\alpha, X_\lambda],\\
[ X_\alpha, DX_\lambda] & = \left(|\lambda|^2 \dim \g{g}_\lambda + \frac{1}{2} a^2_\alpha |\lambda|^2  + c \right)[X_\alpha, X_\lambda].
\end{align*}
On the one hand, if $X_\alpha = 0$ or $X_\lambda = 0$, then~\eqref{equation:derivation} holds for any $c$. On the other hand, if $X_\alpha$, $X_\lambda$ are non-zero vectors, then $[X_\alpha, X_\lambda]\neq 0$ by means of Lemma~\ref{lemma:theta:alpha:new}~(\ref{lemma:theta:alpha:new:3}). Thus, \eqref{equation:derivation} holds for any $X \in \g{g}_\alpha \ominus \R \xi_\alpha$ and any $Y \in \g{g}_\lambda \ominus \R \xi_\lambda$ if and only if $c = -|\alpha|^2$.

Let us focus on the pair $(X, Y) = (X_\nu, \eta_{\alpha, \lambda})$. A straightforward calculation using Lemma~\ref{lemma:brackets}~(\ref{lemma:brackets:1}) leads to 
\begin{equation}\label{proposition:A2:eq1}
a_\nu a_\mu [[[ \theta \xi_\nu, X_\nu ], \xi_\mu], \eta_{\alpha, \lambda} ] = -|\alpha|^2 a^2_\mu [X_\nu, \eta_{\alpha, \lambda}],
\end{equation}
with $(\nu, \mu) \in \{ (\alpha, \lambda), (\lambda, \alpha) \}$. Now, by virtue of Proposition~\ref{proposition:ad}~(\ref{proposition:ad:1})-(\ref{proposition:ad:2}), Proposition~\ref{proposition:rxi:sxi}~(\ref{proposition:rxi:sxi:2})-(\ref{proposition:rxi:sxi:4}), Proposition~\ref{propostion:rs:eta}~(\ref{propostion:rs:eta:2}) and~\eqref{proposition:A2:eq1}, we deduce  
\begin{align}\label{proposition:A2:eq2}
\nonumber D[X_\nu, \eta_{\alpha, \lambda}] & = \left( |\alpha|^2 \dim \g{g}_\alpha + |\lambda|^2 \dim \g{g}_\lambda - \frac{1}{2}a^2_\alpha |\alpha|^2   - \frac{1}{2}a^2_\lambda |\lambda|^2 +c\right)[X_\nu, \eta_{\alpha, \lambda}],\\
[D X_\nu, \eta_{\alpha, \lambda}] & = \left(|\alpha|^2 \dim \g{g}_\alpha + \frac{1}{2} a^2_\mu |\alpha|^2  + c \right)[X_\nu, \eta_{\alpha, \lambda}] - \frac{1}{2}a^2_\mu|\alpha|^2  [X_\nu, \eta_{\alpha, \lambda}], \\
\nonumber [X_\nu, D\eta_{\alpha, \lambda}] & = \left(|\alpha|^2 \dim \g{g}_\alpha + \frac{1}{2} |\alpha|^2  + c \right)[X_\nu, \eta_{\alpha, \lambda}],
\end{align}
with $(\nu, \mu) \in \{ (\alpha, \lambda), (\lambda, \alpha) \}$. According to~\eqref{proposition:A2:eq2}, when $X_\nu= 0$, then~\eqref{equation:derivation} holds for any $c$, with $\nu$ in $\{\alpha, \lambda\}$. On the other hand, when $X_\nu \neq 0$ for some $\nu$ in $\{\alpha, \lambda\}$, then $[X_\nu, \eta_{\alpha, \lambda} ] \neq 0$ by means of Lemma~\ref{lemma:theta:alpha:new}~(\ref{lemma:theta:alpha:new:3}). Hence, \eqref{equation:derivation} holds for any $X \in (\g{g}_{\alpha} \ominus \R \xi_\alpha) \oplus (\g{g}_{\lambda} \ominus \R \xi_\lambda)$ and any $Y  \in \R \eta_{\alpha, \lambda}$ if and only if $c = -|\alpha|^2$. This proves that if $S$ is under the conditions of this Proposition, it is a Ricci soliton.

In order to finish the proof, it suffices to see that $S$ cannot be a Ricci soliton if $M$ is either $SL_3 (\mathbb{H})/Sp_3$ or  $E^{-26}_6/F_4$, provided that $\Phi = \Pi$. Note that in both cases $\dim \g{g}_{\alpha} \geq 4$. Let $X_\alpha$, $Y_\alpha$ be orthogonal unit vectors in $\g{g}_{\alpha} \ominus \R \xi_\alpha$. Thus, $[X_\alpha, Y_\alpha] = 0$ since $2 \alpha$ is not a root. This means that $D[X_\alpha, Y_\alpha] = 0$. Moreover, using $[X_\alpha, Y_\alpha] = 0$ again, together with Proposition~\ref{proposition:ad}~(\ref{proposition:ad:1}), Proposition~\ref{proposition:rxi:sxi}~(\ref{proposition:rxi:sxi:4}) and Lemma~\ref{lemma:brackets}~(\ref{lemma:brackets:2}), we deduce
\begin{equation*}
[D X_\alpha, Y_\alpha] = \frac{1}{2} a_\alpha a_\lambda [[[\theta \xi_\alpha, X_\alpha], \xi_\lambda],Y_\alpha] = \frac{1}{2} a_\alpha a_\lambda [X_\alpha, [[\theta \xi_\alpha, Y_\alpha], \xi_\lambda]] = [ X_\alpha,D Y_\alpha].
\end{equation*}
But $[[[\theta \xi_\alpha, X_\alpha], \xi_\lambda],Y_\alpha] \neq 0$, as follows from combining Lemma~\ref{lemma:nonzero} with Lemma~\ref{lemma:theta:alpha:new}~(\ref{lemma:theta:alpha:new:3}). All in all, we deduce that~\eqref{equation:derivation} does not hold for $X=X_\alpha$ and $Y=Y_\alpha$, and then $S$ is not	 a Ricci soliton.
\end{proof}

\begin{remark}\label{remark:congruency}
Let $M$ be one the following symmetric spaces: $SL_3(\R)/SO_3$, $SL_3(\C)/SU(3)$ or $SO^0_{2,3}/SO_2 SO_3$. Put $\Pi = \{ \alpha_0, \alpha_1 \}$ for the set of simple roots, with $|\alpha_0| \geq |\alpha_1|$. Consider the vector $\xi_{\varphi} = \cos (\varphi) \xi_0 + \sin (\varphi) \xi_1$ for each $\varphi \in [0, \pi/4]$, where $\xi_0$, $\xi_1$ are unit vectors in $\g{g}_{\alpha_0}$ and $\g{g}_{\alpha_1}$, respectively. Define the Lie algebra $\g{s}_\varphi = \g{n} \ominus \R \xi_{\varphi}$ and let $S_\varphi$ be the connected Lie subgroup of $N$ whose Lie algebra is  $\g{s}_\varphi$, for each $\varphi \in [0, \pi/4]$. According to Proposition~\ref{proposition:A2} and to Corollary~\ref{corollary:summary2}, $S_\varphi$ is a Ricci soliton Lie subgroup of $N$. We introduce below the main ingredients to check that $S_{\varphi}$, with $\varphi \in [0, \pi/4]$, constitutes a continuous family of mutually non-congruent Ricci soliton Lie subgroups of the symmetric space $M$. First, note that $\tr (\bar{\Ss}_{\xi_\varphi}^\varphi) = 0$, as follows after some extra considerations from Lemma~\ref{lemma:simplify:ric:S}~(\ref{lemma:simplify:ric:S:3}), where $\bar{\Ss}_{\xi_\varphi}^\varphi$ denotes the shape operator of $S_\varphi$ as a submanifold of $M$ with respect to the unit normal vector $\xi_\varphi$. Now, consider $N_0$ a unit vector proportional to $H_{\alpha_0} + m H_{\alpha_1}$, where $m = 1$ if $|\alpha_0| = |\alpha_1|$ and $m = 3/2$ otherwise, and $N_1$ a unit vector proportional to  $H_{\alpha_0}- H_{\alpha_1}$. Note that $N_0$ and $N_1$ are orthogonal vectors and both normal to $S_\varphi$ in $M$, for all $\varphi \in [0, \pi/4]$, as they belong to $\g{a}$. On the one hand, $\tr (\bar{\Ss}_{N_0}^\varphi)$ does not depend on $\varphi$. However, $\tr (\bar{\Ss}_{N_1}^\varphi)$ does depend on $\varphi$, and as a consequence one can see that the lengths of the mean curvature vectors of $S_{\varphi}$ and of $S_{\psi}$ are different, provided that $\varphi$ and $\psi$ are different and in the interval $[0, \pi/4]$.
\end{remark}

Now, let us see that all the examples corresponding to Corollary~\ref{corollary:summary2}~(\ref{corollary:summary2:3}) are indeed Ricci solitons.

\begin{proposition}\label{proposition:A3}
Let $\Pi =\{\alpha_0, \beta, \alpha_1\}$ be the set of simple roots of a root system of type $A_3$, where $\alpha_0$ and $\alpha_1$ are orthogonal roots. Let $\xi = 2^{-1/2} (\xi_{\alpha_0} + \xi_{\alpha_1})$, where $\xi_\nu \in \g{g}_{\nu}$ is a unit vector for each $\nu \in \{\alpha_0, \alpha_1\}$. Then $S$ is a Ricci soliton.
\end{proposition}

\begin{proof}
Under the assumptions of the statement, we have that  $\dim \g{g}_{\nu} = \dim \g{g}_{\mu}$ for any $\nu$, $\mu \in \Sigma^{+}$,  $a_{\alpha} = a_{\lambda} = 2^{-1/2}$ and $\Phi = \{ \alpha_0, \alpha_1\}$. Consider the decomposition
\[
\g{s} = \R \eta_{\alpha_0, \alpha_1} \oplus (\g{g}_{\alpha_0} \ominus \R \xi_{\alpha_0}) \oplus \g{g}_\beta \oplus (\g{g}_{\alpha_1} \ominus \R \xi_{\alpha_1})  \oplus \g{g}_{\alpha_0 + \beta} \oplus\g{g}_{\beta + \alpha_1} \oplus \g{g}_{\alpha_0 + \beta + \alpha_1}.
\]
Note that $DX$ is proportional to $X$, for any $X$ belonging to one of the subspaces of the above decomposition, by virtue of~\eqref{definition:D}, Proposition~\ref{proposition:ad}~(\ref{proposition:ad:1})-(\ref{proposition:ad:2}), Proposition~\ref{proposition:rxi:sxi}~(\ref{proposition:rxi:sxi:2})-(\ref{proposition:rxi:sxi:3}) and Proposition~\ref{propostion:rs:eta}~(\ref{propostion:rs:eta:2}). Hence, it suffices to check~\eqref{equation:derivation} for $X$, $Y \in \g{s}$ in subspaces of the above decomposition such that $[X,Y] \neq 0$. Let $X_\mu$ be a vector in $\g{g}_\mu \ominus \R \xi_\mu$, for each $\mu \in \Sigma^{+}$, where we assume that $\xi_\nu = 0$ if $\nu \notin \{\alpha_0, \alpha_1\}$. Thus, we need to check~\eqref{equation:derivation} for 
\begin{equation}\label{proposition:A3:eq1}
(X,Y) \in \{ (\eta_{\alpha_0, \alpha_1}, X_\nu), (X_\mu, X_\lambda) \},
\end{equation}
with $ \nu \in \{ \beta, \beta+\alpha_0, \beta+\alpha_1 \}$, $(\mu, \lambda) \in  \{ (\alpha_k, \beta), (\alpha_k, \beta + \alpha_{k+1}) : k=0, 1\}$, and indices modulo~2. Put  $c =- (1/2) |\alpha_0|^2$. Now, we get that~\eqref{equation:derivation} holds for the first pair in~\eqref{proposition:A3:eq1} by virtue of Lemma~\ref{lemma:sum:orthogonal}~(\ref{lemma:sum:orthogonal:2})-(\ref{lemma:sum:orthogonal:3}). If $\dim \g{g}_{\alpha_k} = 1$, then~\eqref{equation:derivation} holds for the second pair in~\eqref{proposition:A3:eq1} trivially since $X_{\alpha_k} = 0$, with $k \in \{0,1\}$. Otherwise, it holds by means of Lemma~\ref{lemma:orthogonal}. 
\end{proof}

Finally, let us complete the proof of the main result of this paper.

\begin{proof}[Proof of the Main Theorem]
Let $S$ be a codimension one Lie subgroup of $N$, where $N$ stands for the nilpotent group of the Iwasawa decomposition of the connected component of the identity of the isometry group of an irreducible symmetric space of non-compact type $M \cong AN$. Recall that $\g{s}$ and $\g{n}$ denote the corresponding Lie algebras of $S$ and $N$, respectively. From Lemma~\ref{lemma:subalgebra:s:xi}, we have $\g{s} = \g{n} \ominus \R \xi$, with 
\begin{equation*}
\xi = \sum_{\gamma \in \Phi} a_{\gamma} \xi_{\gamma},
\end{equation*}
where $\Phi$ is a certain subset of $\Pi$, $\xi_{\gamma}$ is a unit vector of $\g{g}_{\gamma}$ and $a_{\gamma}$ is a positive number, for each $\gamma \in \Phi \subset \Pi$. Hence, any codimension one Lie subgroup $S$ of $N$ determines a subset $\Phi$ of the set of simple roots. According to Corollary~\ref{corollary:summary2}, if $S$ is a Ricci soliton, then $\Phi$ contains one or two roots.

Put first $\Phi = \{ \alpha \}$, for some $\alpha \in \Phi$. Under this assumption, we have achieved a classification of codimension one Ricci soliton Lie subgroups of $N$ in Proposition~\ref{proposition:1:root}. They correspond to the examples from item~(\ref{main:theorem:1}) to item~(\ref{main:theorem:4}) of the Main Theorem except for one case. Note that the case $M = SL_3(\C)/SU_3$ appears in Proposition~\ref{proposition:1:root}~(\ref{proposition:1:root:3}) but not in the Main Theorem~(\ref{main:theorem:4}). This is because there are also examples when $\Phi$ has two roots. Then, it appears in Main Theorem~(\ref{main:theorem:5}).

Now, let us assume that $\Phi$ contains two roots or, in other words, that the unit normal vector $\xi$ to $S$ in $N$ has non-trivial orthogonal projection onto two root spaces associated with simple roots.

On the one hand, if $\Phi$ contains orthogonal roots, then according to Corollary~\ref{corollary:summary2} we must have that: $\Pi$ generates an $A_3$ root system, $\Phi = \{ \alpha, \lambda \}$ with $A_{\alpha, \lambda} = 0$, $a_\alpha = a_\lambda = 1/\sqrt{2}$ and $c = - (1/2) |\lambda|^2$. Conversely, all the examples under these assumptions are Ricci solitons as follows from Proposition~\ref{proposition:A3}. This family of examples corresponds to item~(\ref{main:theorem:6}) of the Main Theorem.

On the other hand, let us assume that $\Phi$ contains two connected roots in the Dynkin diagram. From Corollary~\ref{corollary:summary2}, we get that $\Pi$ generates either an $A_2$ or a $B_2$ root system. Now, using Proposition~\ref{proposition:A2} and Proposition~\ref{proposition:B2} we see that the only examples under these hypotheses are those corresponding to item~(\ref{main:theorem:5}) of the Main Theorem. Note that when $\Phi = \{ \alpha \}$, the examples in $M = SL_3 (\R)/SO_3$ and in $M = SO^{0}_{2,3}/SO_2 SO_3$ are Ricci solitons since they appear also in Main Theorem~(\ref{main:theorem:2}).
\end{proof}

\end{document}